\documentclass[11pt,oneside,english,reqno]{amsart}
\usepackage[T1]{fontenc}
\usepackage{mathrsfs}
\usepackage{amstext}
\usepackage{amsthm}
\usepackage{amssymb}
\usepackage{color}
\usepackage{hyperref}
\usepackage{enumerate}
\usepackage{verbatim} 
\usepackage{enumitem}
\usepackage{multirow}
\usepackage{dirtytalk}
\usepackage{graphicx}
\usepackage{times}
\usepackage{wrapfig}
\usepackage{lipsum}

\makeatletter
\numberwithin{equation}{section}
\numberwithin{figure}{section}
\theoremstyle{plain}
\newtheorem{thm}{\protect\theoremname}
\theoremstyle{definition}
\newtheorem{defn}[thm]{\protect\definitionname}
\theoremstyle{remark}
\newtheorem{rem}[thm]{\protect\remarkname}
\theoremstyle{plain}
\newtheorem{lem}[thm]{\protect\lemmaname}
\theoremstyle{plain}
\newtheorem{prop}[thm]{\protect\propositionname}
\theoremstyle{plain}
\newtheorem{cor}[thm]{\protect\Corollaryname}
\theoremstyle{plain}
\newtheorem{hyp}[thm]{\protect\hypothesisname}
\theoremstyle{plain}
\newtheorem{ex}[thm]{\protect\examplename}
\theoremstyle{plain}

\makeatother

\usepackage{babel}

\providecommand{\definitionname}{Definition}
\providecommand{\lemmaname}{Lemma}
\providecommand{\propositionname}{Proposition}
\providecommand{\remarkname}{Remark}
\providecommand{\theoremname}{Theorem}
\providecommand{\hypothesisname}{Hypothesis}
\providecommand{\examplename}{Example}
\providecommand{\notationname}{Notation}
\providecommand{\Corollaryname}{Corollary}

\newcommand{\la}{\langle}
\newcommand{\ra}{\rangle}
\newcommand{\cA}{\mathcal{A}}

\newcommand{\cD}{\mathcal{D}}
\newcommand{\cE}{\mathcal{E}}
\newcommand{\cF}{\mathcal{F}}

\newcommand{\cK}{\mathcal{K}}

\newcommand{\cN}{\mathcal{N}}

\newcommand{\cS}{\mathcal{S}}

\newcommand{\cW}{\mathcal{W}}

\newcommand{\CC}{\mathbb{C}}

\newcommand{\EE}{\mathbb{E}}
\newcommand{\FF}{\mathbb{F}}

\newcommand{\NN}{\mathbb{N}}

\newcommand{\PP}{\mathbb{P}}

\newcommand{\RR}{\mathbb{R}}

\newcommand{\bX}{\mathbf{X}}
\newcommand{\bY}{\mathbf{Y}}

\newcommand{\dd}{\mathop{}\!\mathrm{d}}

\allowdisplaybreaks
\usepackage[T1]{fontenc}
\usepackage[utf8]{inputenc}

\usepackage{mathtools}
\let\oldtocsection=\tocsection
\let\oldtocsubsection=\tocsubsection
\let\oldtocsubsubsection=\tocsubsubsection
\renewcommand{\tocsection}[2]{\hspace{0em}\oldtocsection{#1}{#2}}
\renewcommand{\tocsubsection}[2]{\hspace{1em}\oldtocsubsection{#1}{#2}}
\renewcommand{\tocsubsubsection}[2]{\hspace{2em}\oldtocsubsubsection{#1}{#2}}
\usepackage{stmaryrd}
\numberwithin{thm}{section}

\usepackage[table]{xcolor}% http://ctan.org/pkg/xcolor
\usepackage{dsfont}

\usepackage{enumitem}
\setlist[description]{leftmargin=\parindent,labelindent=\parindent}

\usepackage{shuffle}

\title[Universal approximation on non-geometric rough paths]{Universal approximation on non-geometric rough paths and applications to financial derivatives pricing}

\author{Fred Espen Benth$^{1,2}$}
\address{$^1${\tiny Department of Data Science and Analytics, BI Norwegian Business School.}}
\address{$^2${\tiny Department of Mathematics, University of Oslo.}}

\author{Fabian A. Harang $^3$}
\address{$^3${\tiny Department of Economics, BI Norwegian Business School.}}

\author{Fride Straum$^4$}
\address{$^4${\tiny Department of Mathematical Sciences, NTNU, Trondheim, Norway.}}

\thanks{We acknowledge the support from the Center for Advanced Study (CAS) in Oslo, Norway, which funded and hosted our research project "Signatures for Images" during the 2023/24 academic year. We are grateful to an anonymous referee and the editor for constructive critics leading to a significantly improved version of our paper.}

\begin{document}

\maketitle
\makeatletter
\@setaddresses
\makeatother
\begin{abstract}
We present a novel perspective on the universal approximation theorem for rough path functionals, introducing a polynomial-based approximation class. We extend universal approximation to non-geometric rough paths within the tensor algebra. This development addresses critical needs in finance, where no-arbitrage conditions necessitate Itô integration. Furthermore, our findings motivate a hypothesis for payoff functionals in financial markets, allowing straightforward analysis of signature payoffs proposed in \cite{arribas2018derivativespricingusingsignature}.
\end{abstract}

\tableofcontents

\section{Introduction}

Financial derivatives 
%are vital instruments that 
allow firms to hedge exposure to diverse financial risks, such as price, currency, or interest rate fluctuations. In commodity markets, participants contend with uncertainties stemming from production volumes, demand variability, and transportation logistics.  To address these challenges, a variety of financial derivatives are traded, designed to mitigate risk and stabilize revenues. These instruments often involve complex transactions with multivariate payoffs that replicate potential future cash flows. For example, power producers manage both price and volume risks, while renewable energy producers face weather-driven uncertainties, such as wind or cloud cover. Retailers, on the other hand, contend with temperature-sensitive demand variations. Multivariate derivatives, such as spread and quanto options, are particularly attractive for managing such intertwined risks, although their pricing remains a significant challenge due to the lack of closed-form solutions.

Pricing these products typically relies on numerical approximations of theoretical prices using techniques like Monte Carlo simulations or numerical solutions of partial differential equations. However, the inherent complexity and high dimensionality of such problems demand novel mathematical tools to improve efficiency and accuracy.

In recent years, insights from the theory of rough paths, introduced by Lyons \cite{Lyons1998}, have emerged as a powerful framework for understanding stochastic processes and their functionals. A central concept in this theory is the 
signature, a collection of iterated integrals that captures the fundamental characteristics of a path. Initially studied by %K. T. 
Chen \cite{KTChen54} from an algebraic topology perspective, the signature framework has since evolved into a dynamic field of research, offering new methodologies for modelling and computation in various domains.
More precisely,  for a smooth path $X:[0,T]\rightarrow \RR^d$ we define $\dd X_t = \frac{\dd}{\dd t} X_tdt$, and for $s\leq t$
\begin{equation}\label{iterated integral}
    \bX^{n}_{s,t}:=\int_{s<r_1<\dots<r_n<t} \bigotimes_{i=1}^n \dd X_{r_i}\in (\RR^d)^{\otimes n}. 
\end{equation}
For $p\in \NN$, we define the $p$-truncated signature $\bX^{\leq p}$ by 
\begin{equation*}
    \bX^{\leq p}_{s,t} = \left(1,\bX^{1}_{s,t},\ldots,\bX^{p}_{s,t}\right) \in \{\mathbf{1}\}\otimes \bigotimes_{i=1}^p (\RR^d)^{\otimes i} = T^{p}(\RR^d) 
\end{equation*}
and the  {\em signature} $\bX^{\leq \infty}$ as the limit when $p\rightarrow \infty $. The space $T^{p}(\RR^d)$ is known as the truncated tensor algebra, and we call $T((\RR^d))$ given formally as $T^{\infty}(\RR^d)$, the (extended) tensor algebra. A more detailed introduction will be provided in Section \ref{sec:basics}. 
Beyond smooth paths, the theory extends to stochastic paths, where probabilistic tools enable the construction of these integrals under frameworks like It\^o or Stratonovich integration (see e.g. \cite{Friz_Victoir_2010}).

In many ways, the iterated integral signature can be seen as an infinite-dimensional extension of polynomials. In fact, it shares many of the key features of polynomials, but related to one-parameter functions. Some key properties are 
\begin{enumerate}[label=(\roman*), leftmargin=*]
    \item Chen's relation holds, i.e., the signature is a multiplicative functional. More specifically, for any $u\in [s,t]$ then $\bX_{s,u}\otimes \bX_{u,t}=\bX_{s,t}$. The tensor product is understood in the extended tensor algebra $T((\RR^d))$, which is defined in Section \ref{sec:basics}. This not only provides computational efficiency but also serves as a fundamental building block in rough integration theory. 
    \item The signature uniquely characterizes the path up to tree-like equivalence \cite{hamblylyons2010}. 
    \item  It is re-parametrization invariant; for a monotone increasing function $\phi:[0,T]\rightarrow[0,T]$ with $\phi(0)=0$ and $\phi(T)=T$, and define $\bar{X}_t = X_{\phi(t)}$, then $\bar{\bX}^{\leq \infty}_{0,T}= \bX^{\leq \infty}_{0,T}$ \cite{Friz_Victoir_2010}. 
    \item The signature is invariant to translation; If for some $a\in \RR^d$ we have $Y_t= X_t+a$ for all $t\in[0,T]$, then $\bY = \bX$. 
    \item It is associated with a rich algebraic structure (see e.g. \cite{Friz_Victoir_2010} for an introduction). 
    \item The signature is tightly connected to stochastic integration theory, and naturally encode information about stochastic integration choices, enriching the understanding of stochastic integration theory beyond the classical martingale theory \cite{frizhairer}. 
    \item The signature characterizes the law of stochastic processes under sufficient signature moment decay conditions \cite{ChevyrevOberhauser22}\footnote{One should note that the decay conditions derived in \cite{ChevyrevOberhauser22} lead to rather strong assumptions that are difficult to verify, and excludes several interesting examples, such as the geometric Brownian motion.}. 
\end{enumerate}
 Importantly, the signature serves as a universal approximation basis for continuous functionals on path space, akin to how polynomials approximate real-valued functions. This universal approximation property allows any continuous functional on the space of Lipschitz paths to be approximated arbitrarily well by a linear functional of the signature \cite{cuchiero2023universal}. The universal approximation property is proven through the Stone-Weierstrass theorem, and thus requires a feature set that forms a sub-algebra of the continous functionals on path-space. 

The signature framework extends to stochastic functional approximation. A key challenge lies in encoding the choice of stochastic integration—such as It\^o or Stratonovich—into the functional representation. While the universal approximation property has been extended to geometric rough paths, which naturally align with Stratonovich integration, this leaves a significant gap for It\^o-based financial functionals, which are most prevalent in practice.

This article addresses the challenges of applying signature-based universal approximation to non-geometric rough paths, with a focus on practical implementation in financial markets. By bridging the gap between non-geometric rough paths and universal approximation, we investigate the framework for efficient pricing in the It\^o setting for  complex financial derivatives, providing several examples throughout the text. The remainder of the paper develops the theoretical foundations and explores applications in detail.

\subsection{Main ideas and contribution}

In this  article we present a universal approximation result for functionals of non-geometric rough paths. The main challenge with non-geometric rough paths and by extension the non-geometric signature, is that the multiplication of two elements in the non-geometric signature does not yield another element contained in the same signature. This implies that just considering the linear span of signature elements as the subset of continuous linear functionals to use for functional approximation is not sufficiently rich to become an algebra; a strict requirement of the Stone-Weierstrass theorem. 
To overcome this challenge, we enrich the feature set, the linear span of our signature terms, by   polynomials of the signature terms. This becomes a very large class of features that we use for universal approximation, but provides the sufficient set which guarantees universal approximation in the non-geometric setting. When applying this theorem to {\em geometric} rough paths, the approximation can be written as a linear functional acting on the signature through the shuffle product. 

In finance the universal approximation theorem with signatures has been successfully used in the context of pricing complex financial derivatives, see \cite{arribas2018derivativespricingusingsignature,Arribasetal20,LyonsNejadAribas2019}. Such derivatives typically have a payoff of functional form, in practice often a so-called \say{Asian} structure.  There the payoff depends not on the price at a given time, but on the average price over a time period, thus introducing an integral $\int_a^b X_r \dd r$ with the price being $t\mapsto X_t$. For even more complicated structures, one can encounter derivatives in energy markets, where the payoff is given as an integral over the product of two stochastic processes, $\int_a^b X_rY_r \dd r$, for example representing electricity price and temperature. In even more complicated models, one could imagine compositions of a finite number of such structures. The point is that payoff functionals, mapping paths to prices $X\mapsto F(X)$, are typically given in a very specific form of (lower-ordered) signature functionals. In the examples above, the payoff functional can be seen as continuous function $f$ acting on a finite number of terms in the signature of the price signals, enriched with a time component. 

This motivates a working hypothesis of the paper; we consider functionals that we assume to be given as continuous functions acting on terms from the signature. While being a relevant hypothesis in many practical applications, one can resort to classical function approximation techniques to approximate the continuous function of interest. Approximating this function through some polynomial can under various assumptions yield convergence rates for the approximation, providing theoretical guarantees important for practical implementation. This simplifies the derivatives pricing problem from \cite{arribas2018derivativespricingusingsignature} where machine learning techniques are suggested to find the functional approximation, resulting in an approximation which is difficult to analyze from a practical perspective. In contrast, we believe that our approach, mixing the new universal approximation method, with the hypothesis suited for payoff functionals provides a simple way of using signature features for practical pricing problems arising in financial markets. 

Throughout the article, we illustrate our results and contributions through examples, with an emphasis on derivatives pricing in energy markets where there exists many complex derivatives structures. %, for which the proposed method really shines. 
We emphasise %is also to provide 
deriving rather explicit conditions ensuring  convergence rates for the approximations that we introduce.% can be obtained. 

\subsection{Organization of the paper}

In Section \ref{sec:basics} we provide a basic introduction of the signature and building the algebraic foundation for our analysis. As this article is targeting an audience in financial mathematics, we have chosen to provide a detailed introduction to the algebraic side. %, but this section can be skipped for the readers well versed in signatures and rough paths. 
In Section \ref{sec:universal} we recall the state of the art in universal approximation, and propose the new universal approximation for non-geometric rough paths. In Section \ref{sec:stochastics} we discuss stochastic price paths and the computation of signature correlators. Section \ref{sec:derivatives} combines our considerations and provides some approximation results for financial derivatives. This is highlighted with a discussion on the applications in energy markets. 
At last we provide a conclusion with an outlook to future developments in Section \ref{sec:conclussion}.

\subsection{Notation}\label{sec:notation}
For a complete metric space $(E,d_E)$, we denote by $C([0,T];E)$ the space of continuous paths $X:[0,T]\rightarrow E$ with the uniform topology. The set of continuous paths with finite $p\geq 1$ variation from $[0,T]$ into $E$ is denoted by $C^{p-var}([0,T];E)$, and is equipped with the seminorm 
\begin{equation*}
    \| X\|_{p,[0,T]} := \left( \sup_{P\in \cD} \sum_{t_k\in P} d_E(X_{t_k},X_{t_{k+1}})^p\right)^{\frac{1}{p}}, 
\end{equation*}
where $\cD$ is the collection of all partitions over $[0,T]$. 
Whenever the interval $[0,T]$ under consideration is otherwise clear, we simply write $ \| X\|_{p}= \| X\|_{p,[0,T]}$.
Recall that when $x\in\mathbb N\setminus \{0\}$, the gamma function $\Gamma$ satisfies $\Gamma(x+1)=x!$, and by slightly abuse of notation we will write $x!:=\Gamma(x+1)$ for $x>0$. 
Throughout the article, we will occasionally consider Hilbert spaces, denoted by $H$, and where the inner product is then given by  $\la \cdot,\cdot\ra _H$ and the associated norm is denoted by $|\cdot |_H$.  If the space $H$  under consideration is clear, we dismiss the index and simply write $\la\cdot,\cdot \ra $ and $|\cdot |$ for the inner product and norm.  
For a path $X:[0,T]\rightarrow \RR^d$ we define the time-enhanced path $\hat{X}:[0,T]\rightarrow \RR^{d+1}$, where   $ \hat{X}_t=(t,X_t)$. This notation will be used consistently throughout the article. Frequently, we resort to the notation $\triangle_T=\{(s,t)\in [0,T]^2\mid 0\leq s\leq t\leq T\}$.

\section{Basics of words and signatures}\label{sec:basics}

In this section, we will provide a fundamental overview  of the conventions and concepts related to weighted words, denoted by $\pi$, and signatures, represented by $\bX^{\leq\infty}$, along with their pairing $\la \pi, \bX^{\leq\infty}\ra$.

\subsection{Words}
Given that we are working with $\RR^d$-valued paths $X:[0,T]\rightarrow \RR^d$, the alphabet of our consideration is $\cA=\{1,\dots,d\}$. Throughout this article, $\cA$ will denote the set $\{1,\dots, d\}$ unless otherwise stated.
\begin{defn}
    A \textbf{word} of length $n\in\NN$ is a sequence $w=\mathbf{i}_1\dots\mathbf{i}_n$ where $\mathbf{i}_j\in \cA$ for every $j=1,\dots, n$. We denote by $\cW_n$ the set of all words of length $n\in\NN$. For $n=0$ we have that $\cW_0=\{\emptyset\}$ where $\emptyset$ is the \textbf{empty word}. We further let $\cW$ denote the set of all words.
\end{defn} 
The algebra $\RR\langle\cA\rangle$ is introduced as  $T(\text{vect}_\RR(\cA))$, representing the vector space generated by $\cW$.

\begin{defn} The algebra of all non-commutative polynomials in $\cA$ is defined to be
    \[
\RR\la \cA\ra=\left\{\pi=\sum_{w\in \cW}\alpha_ww\mid \alpha_w\in\RR,\; \alpha_w\neq 0\text{ for a finite number of }w\in \cW\right\}.
\] 
We refer to $\pi\in \RR\la \cA\ra$ as a \textbf{weighted word.}
\end{defn} 
$\RR\la \cA\ra$ forms an algebra with respect to \textbf{concatenation} which for two words $w=\mathbf{i}_1\dots \mathbf{i}_n$ and $w'=\mathbf{j}_1\dots \mathbf{j}_m$ is given by $w\cdot w'=ww'=\mathbf{i}_1\dots \mathbf{i}_n\mathbf{j}_1\dots \mathbf{j}_m.$ By letting concatenation be bilinear, we can define concatenation of weighted words. See Example \ref{Ex:concatenation} below.

\begin{ex}\label{Ex:concatenation}
    Suppose we work with an alphabet given by $\cA=\{\mathbf{a,b,c}\}$.  Then we have that \[\cW=\{\emptyset,\mathbf{a,b,c,aa,ab,ac},\dots\},\] where $\mathbf{ab}\neq \mathbf{ba}$. As an example, one element $\pi\in \RR\langle \cA\rangle$ is \[\pi=  \emptyset+2\mathbf{a}+\sqrt{3}\mathbf{ab}+\cdots+100 \mathbf{abcbac}.\] Moreover, the concatenation of the weighted words $(3\mathbf{ab+a})$ and $(\mathbf{b+c})$ is \[(3\mathbf{ab+a})\cdot(\mathbf{b+c})=3\mathbf{abb}+3\mathbf{abc+ab+ac}.\]
    Note that we use alphabets consisting of the natural numbers up to dimensions $d$ in this paper. However, to separate clearly between the weights $a_w$ and the words $w$, we used letters for the alphabet in this example.
\end{ex}

\begin{rem}\label{remrecursive}
    In the next Subsection, we will introduce the concept of signatures. However, it is worth noting that the signature $\bX^{\leq\infty}$ of a path $X:[0,T]\rightarrow \RR^d$ can also be defined recursively through the projections of words. 
    Since the path is $\RR^d$-valued, let $\cW$ denote the collection of words formed from the alphabet $\cA = \{1,\dots,d\}$. We can regard the signature $\bX^{\leq\infty}$ as an element of $C(\triangle_T)^\cW$. To formalize this, we define it recursively using the projections $\la w, \bX^{\leq\infty}\ra$ for $w \in \cW$, as follows:

    First, for the empty word, we set $\la \emptyset, \bX^{\leq\infty}_{s,t} \ra := 1$. For any non-empty word $w = \mathbf{i}_1 \dots \mathbf{i}_n \in \cW \setminus \{\emptyset\}$, we define the projection as:
    \[
    \la w, \bX^{\leq\infty}_{s,t} \ra := \int_s^t \la \mathbf{ i}_1 \dots \mathbf{i}_{n-1}, \bX^{\leq\infty}_{s,r} \ra \, dX^{\mathbf{i_n}}_r.
    \]
\end{rem}

\subsection{Signatures}
In this Subsection, we introduce the basic concepts of signatures. We begin by presenting key algebras and a Hilbert space that plays a central role in the theory of signatures. Next, we define what a signature is, and finally, we demonstrate how to pair a weighted word with a signature. 
Let us start with the setup.
Let $H$ denote a general $\FF$-Hilbert space, where the field $\FF$ is either $\RR$ or $\CC$. We need to introduce a triplet of spaces $T(H)\subset \cF(H)\subset T((H))$. 
First of all, by convention we let $H^{\otimes 0}=\FF,$ and throughout this article $\otimes$ denotes the Hilbert space tensor product. 
\begin{defn}
    For a Hilbert space $H$ we call \[T((H))=\prod_{n=0}^\infty H^{\otimes n}=\{x=(x_0,x_1,\dots,x_n,\dots)\mid x_n\in H^{\otimes n},\, \text{for}\,n=0,1,2,...\}. \] the \textbf{extended tensor algebra}.
\end{defn} As we will see later, any well-defined signature naturally belongs to the extended tensor algebra $T((H))$. Another space closely linked to the theory of words and signatures is the tensor algebra, which we introduce next: 
\begin{defn}
    The \textbf{tensor algebra} is given by the algebraic direct sum \[T(H)=\bigoplus_{n=0}^\infty H^{\otimes n}=\{x\in T((H))\mid\,\forall\; x\;\exists\; N\in\NN \text{ such that }x_n=0\;\forall n\geq N\}. \] 
\end{defn}
We equip $T((H))$ with a sum $+$, scalar multiplication, and a product $\otimes$. These operations are defined as follows:  For elements $x=(x_0,x_1,\dots,x_n,\dots)$ and $y=(y_0,y_1,\dots,y_n,\dots)$ in $T((H))$, we define the the sum of $x$ and $y$ element-wise, i.e.,  \[x+y=(x_0+y_0,x_1+y_1,\dots,x_n+y_n,\dots).\] Scalar multiplication is also defined element-wise as \[\lambda x=(\lambda x_0,\lambda x_1,\dots),\quad\lambda\in \RR.\] Lastly, the product of $x$ and $y$ is given by 
\[
x\tilde\otimes y = (z_0,z_1,\dots,z_n,\dots)\quad \mathrm{where} \quad z_n=\sum_{k=0}^nx_k\otimes y_{n-k}.
\] 
These operations turn the tensor algebra $T(H)$ into an algebra. Moreover, we denote by \[T^N(H)=\{x\in T(H)\mid x_n=0\;\forall n> N\}\]
the $N$-truncated tensor algebra. It becomes an algebra with respect to $\tilde\otimes$.

\begin{rem}\label{rem1} We present next two useful facts about the tensor algebra and the extended tensor algebra when $H=\RR^d$. First, we have the isomorphisms
\[T(\RR^d)\simeq T((\RR^d)^*)\simeq \RR\langle\cA\rangle,\] see e.g. \cite{Lang2002}. Thus, every word $\pi\in \RR\la\cA\ra$ can be viewed as an element in $T(\RR^d).$ Secondly,
we have a dual pairing between $T(\RR^d)$ and $T((\RR^d))$, see e.g. \cite{Lyons2007}, because of the fact that 
\[
T((\RR^d)^*)\simeq T((\RR^d))^*
\] where $T((\RR^d))^*$ denotes the algebraic dual of $T((\RR^d))$. Hence for $\pi\in T(\RR^d)$ and $x\in T((\RR^d))$ we denote by $\la \pi,x\ra$ the algebraic dual pairing between $\pi$ and $x$.

Putting these two facts together, we find that
    \[
    \RR\la \cA\ra \simeq T((\RR^d))^*
    \]
    and therefore every weighted word $\pi\in \RR\la \cA\ra$ can be viewed as a linear functional on the extended tensor algebra, $T((\RR^d))$.
\end{rem}

Now we want to introduce a Hilbert space for which a big class of signatures belongs to, called the full Fock space. For $x=x_1\otimes \dots\otimes x_n$, $y=y_1\otimes \dots\otimes y_n\in H^{\otimes n}$ we consider an inner product on $H^{\otimes n}$ by \[\langle x,y\rangle_n=\prod_{i=1}^n\langle x_i,y_i\rangle_H.\]
Hence, the norm on $H^{\otimes n}$ becomes $|x|_n^2=\prod_{i=1}^n|x_i|^2_H$. 
Then for $x=(x_n),y=(y_n)\in \bigoplus_{n=0}^\infty H^{\otimes n},$ where $H^{\otimes 0}=\FF,$ we can define an inner product by\[
\langle x,y\rangle_{\cF(H)}=\sum_{n=0}^\infty \langle x_n,y_n\rangle_n, \]
and hence a norm \[\|x\|_{\cF(H)}^2=\sum_{n=0}^\infty |x_n|_n^2.\] The \textbf{full Fock space} over $H$ is the Hilbert space given by the topological direct sum, $\widehat{\bigoplus}$, 

\[
\cF(H):=\widehat{\bigoplus}_{n=0}^\infty H^{\otimes n}=\left\{x=(x_n)\mid x_n\in H^{\otimes n},\, n\in\NN,\,\|x\|_{\cF(H)}^2<\infty \right\}.
\]Moreover, $T(H)$ is a dense subspace in $\cF(H)$. 
\begin{ex}
\label{ex:word-functional-identification}
    Recall that for $\cA=\{1,\dots,d\}$ we have that $\RR\la\cA\ra\simeq T(\RR^d)$ and therefore any $\pi\in\RR\la\cA\ra $ can be recognized as an element $e_\pi\in T(\RR^d)$. Moreover, since $T(\RR^d)$ is a (dense) subspace of $\cF(\RR^d)$, we can compute the Fock norm of a weighted word $\pi\in \RR\la\cA\ra.$  For example, let $d=2$ so that $\cA=\{1,2\}$ and consider the weighted word
    \[\pi = 2\emptyset + 3\cdot \mathbf{1}+\mathbf{12}.\] Then \[e_\pi=2e_\emptyset+3e_1+e_1\otimes e_2=(2e_\emptyset,3e_1,e_1\otimes e_2)\] 
    where $e_{\emptyset}=1$ is the basis vector in the field $\mathbb R$ while
    $e_1,e_2\in\mathbb R^2$ are the basis vectors in $\mathbb R^2$. Hence, we get that
    \[ \|\pi\|^2_{\cF(\RR^d)}:=\|e_\pi\|^2_{\cF(\RR^d)}=2^2|e_\emptyset|^2+3^2|e_1|^2+|e_1|^2|e_2|^2=4+9+1=14.
    \]
\end{ex}

Next, we recall the definition of a multiplicative functional, a fundamental object in the theory of rough paths \cite{Lyons2007}:
\begin{defn}\label{def:p-var}
    A \textbf{multiplicative functional} of degree $N\in\NN$ is a continuous map \[\triangle_T\ni (s,t)\mapsto \mathbf{X}_{s,t}=(1, X^1_{s,t},\dots, X^N_{s,t})\in T^N(H)\] satisfying \textit{Chen's identity:}
\[\mathbf{X}_{s,t}=\mathbf{X}_{s,u}\tilde\otimes \mathbf{X}_{u,t}\quad \forall\; 0\leq s\leq u\leq t\leq T.\]
\end{defn}

For $p\geq1$, a multiplicative functional $\mathbf{X}=(1,X^1,\dots,X^N)$ of degree $N\in\NN$ is said to have \textbf{finite $p$-variation} if  
\[
\|\mathbf{X}\|_{p,[0,T]}:=\max_{0\leq n\leq N}\sup_{P\in\cD}\sum_{t_k\in P}|X^n_{t_k, t_{k+1}}|_n^{p/n}<\infty
\]
and $\cD$ is the collection of all finite partitions over $[0,T]$, and $|\cdot|_n=\prod_{i=1}^n|\cdot|_H$ is the canonical Hilbert tensor norm on $H^{\otimes n}$. \footnote{Note that this tensor norm is not properly defined, but is well-defined for pure tensors and extends linearly for arbitrary elements.}

\begin{defn}\label{def:rough path}
    Let $p\geq 1$ and let $\mathbf{X}_{s,t}=(1,X^1_{s,t},\dots,X^{\lfloor p\rfloor} _{s,t})$ be a multiplicative functional of degree $\lfloor p\rfloor \in\NN$. We call $\mathbf{X}$ a \textbf{$p$-rough path} if $\|\mathbf{X}\|_{p,[0,T]}<\infty.$
\end{defn}

The set of $p$-rough paths becomes a complete metric space with $d(\bX,\bY)=\|\bX-\bY\|_{p,[0,T]}$.

\begin{thm}\label{thm:lyons-extension}[Lyons' extension theorem \cite{Lyons2007}]
    Let $\mathbf{X}$ be a p-rough path. Then for any $n\geq \lfloor p\rfloor +1$ there exists a unique continuous map
    \[X^n:\triangle_T\rightarrow H^{\otimes n}\] such that \[\bX^{\leq \infty}=(1,X^1,\dots,X^{\lfloor p\rfloor},\dots,X^n,\dots):\triangle_T\rightarrow T((H))\] is a multiplicative functional with finite $p$-variation. We call $\bX^{\leq \infty}$ the \textbf{signature} of $\mathbf{X}$. Moreover, we have that \[|X^k_{s,t}|_k\leq \frac{\|\mathbf{X}\|_{p,[s,t]}^k}{\beta(p)(k/p)!}\] where $|\cdot|_k$ is the canonical Hilbert space norm on $H^{\otimes k}$ and $\beta(p)$ is a constant only dependent of $p$.
\end{thm}

\begin{defn}
    We denote by
    \[V^p(\triangle_T,T^{\lfloor p\rfloor}(\RR^d))=\left\{\mathbf{X}^{\leq p}:\triangle_T\rightarrow T^{\lfloor p\rfloor}(\RR^d)\mid \mathbf{X}^{\leq p}\text{ is a $p$-rough path}\right\}.\] Moreover, the well-defined map $S:V^p(\triangle_T,T^{\lfloor p\rfloor}(\RR^d))\rightarrow V^p(\triangle_T,T((\RR^d)))$ given by\[S(\mathbf{X}^{\leq p})=\mathbf{X}^{\leq \infty}, \]is called the \textbf{signature map}.
\end{defn}

\begin{defn}[Geometric rough paths, \cite{Lyons2007}]\label{def:geometric} A \textbf{geometric $p$-rough path} is a $p$-rough path that can be expressed as a limit of $1$-rough paths in the $p$-variation distance. The space of geometric $p$-rough paths in $H$ is denoted by $G^p(H)$.
    
\end{defn}

We now present a useful theorem: the signature of a $p$-rough path is an element of the Fock space, instead of the entire extended tensor algebra, $T((H))$.  To this end, let \[E_{\alpha,\beta}(z)=\sum_{n=0}^\infty \frac{z^n}{(\alpha n+\beta)!},\quad \alpha,\beta\in \CC,\; \mathfrak{R}(\alpha),\mathfrak{R}(\beta)>0,\;z\in \CC\] be the \textit{Mittag-Leffler function}.

\begin{cor}\label{cor1}
    For $p\geq1$ let $\mathbf{X}^{\leq p}$ be a $p$-rough path. We have that $\bX^{\leq \infty}_{s,t}\in\cF(H)$ for any $(s,t)\in \triangle_T$.
\end{cor}
\begin{proof} 
From Lyons' extension theorem \ref{thm:lyons-extension}, we have that $\mathbf X^{\leq\infty}_{s,t}\in T((H))$, and, moreover, 
\begin{align*}
        \|\bX^{\leq \infty}_{s,t}\|_{\cF(H)}=\sum_{n=0}^\infty |X^n_{s,t}|_n \leq \sum_{n=0}^\infty \frac{\|\mathbf{X}\|_{p,[s,t]}^n}{\beta(p)(n/p)!} \leq  \frac{E_{1/p,0}(\|\mathbf{X}\|_{p,[s,t]})}{\beta(p)} <\infty.
    \end{align*}
    The result follows. 
\end{proof}

We wish to utilize the dual pairing in Remark \ref{rem1} between weighted words $\pi\in \RR\la \cA\ra$ and signatures $\bX^{\leq\infty}_{s,t}\in T((\RR^d))$. Hence, from now on we consider the case when $H=\RR^d$. Given an $\RR^d$-valued path $\mathbf X$ for which the signature $\bX^{\leq \infty}_T:=\bX^{\leq \infty}_{0,T}\in T((\RR^d))$ is well-defined, and a word $w=\mathbf{i}_1\dots \mathbf{i}_n\in\cW_n$, we can identify $w$ as an element $e_w\in T(\RR^d)\simeq T((\RR^d))^*$ and we have the following symbolic definition of the dual pairing of a word and a signature
\begin{equation}
\label{basic-def-word-sig}
\langle w,\bX^{\leq \infty}_T\rangle:=\langle e_w,\bX^{\leq \infty}_T\rangle
\end{equation}
where $e_w =  e_{\mathbf{i}_1}\otimes \dots \otimes e_{\mathbf{i}_n}$, and $e_{n}$ is the $n$th unit vector in $\RR^d$.    
If the underlying path $X$ is continuous and of bounded variation, then the right hand side of \ref{basic-def-word-sig} is well-defined in the iterated Riemann-Stieltjes sense as follows 
\begin{equation}
   \langle w,\bX^{\leq \infty}_T\rangle =\underset{0<u_i<\dots<u_n<T}{\int\dots\int}dX_{u_1}^{\mathbf{i_1}}\dots dX_{u_n}^{\mathbf{i_n}}.
\end{equation}

Here and in the sequel of this paper we shall use the generic notation $\pi$ to signify a (finite) linear combination of elements $e_w$, or, after identification, a finite linear combination of words (as in the Example \ref{ex:word-functional-identification} above). Indeed, if $\pi=\sum_{w\in\cW}a_ww$ for $w\in\cW$ and $a_w\in\mathbb{R}\setminus\{0\}$ for finitely many $w\in\cW$, then,
\begin{equation*}
\langle\pi,\bX^{\leq \infty}_T\rangle=\sum_{w\in\cW}a_w\langle w,\bX^{\leq \infty}_T\rangle,
\end{equation*}
with $\langle w,\bX^{\leq \infty}_T\rangle$ given as in \eqref{basic-def-word-sig}.

The shuffle product is turning $T(\mathbb R^d)$ into an algebra, and is very convenient and useful in operating with products of signatures. Indeed, the shuffle product is the key operation in the rough path theory that linearises nonlinear functionals, at least approximately, as can be seen in the 
Universal Approximation Theorem (as we recall for the  convenience of the reader in Proposition \ref{prop:universal approximation}).  We define the shuffle product $\shuffle$ next: The shuffle product between two words gives a weighted word which is constructed by taking a linear combination of all the different ways to combine two words while preserving their own orders. As an example, if we shuffle together the words $\mathbf{ab}$ and $\mathbf{c}$, we get 
\[
\mathbf{ab}\shuffle \mathbf{c=abc+acb+cab}.
\]
Note that $\mathbf{bac}, \mathbf{bca}$ and $\mathbf{cba}$ is  not in the sum on the right hand side since it violates the order of $\mathbf{ab}$.

Below follows an important Lemma for products of signatures and the shuffle product: 

    \begin{lem}[Shuffle property, \cite{Lyons2007}]\label{lemma:shuffle}
    Let $\pi,\pi'\in T(\RR^d)$ and let $\mathbf{X}$ be a geometric rough path, according to Definition \ref{def:geometric}. Then the following identity holds
    \[
    \langle\pi,\bX^{\leq \infty}_{0,T}\rangle\langle\pi',\bX^{\leq \infty}_{0,T}\rangle=\langle\pi\shuffle \pi',\bX^{\leq \infty}_{0,T}\rangle.
    \]
\end{lem}

Interestingly, one can link shuffle product monomials to monomials of the signatures: given $\pi\in T(\RR^d)$ we define  $\pi_n\in T(\RR^d)$ for every $n\in \NN$ by
\[
\pi_0=\pi_\emptyset,\quad\pi_1=\pi, \quad\pi_2=\pi_1\shuffle \pi_1,\quad\pi_n=\pi_{n-1}\shuffle\pi_1=\overbrace{\pi_1\shuffle \dots\shuffle\pi_1 }^{\text{n-times}}.
\]
Then by the shuffle property we have that 
\[
\langle \pi_n,\bX^{\leq \infty}_{0,T}\rangle=\langle \pi_{n-1},\bX^{\leq \infty}_{0,T}\rangle\langle \pi_1,\bX^{\leq \infty}_{0,T}\rangle=\langle \pi,\bX^{\leq \infty}_{0,T}\rangle ^n.
\]\label{shuffletrick}
This gives a convenient link between monomials of signatures and monomials of words.

\begin{rem}
    Note that the shuffle property can be seen as an extension of the \say{integration by parts} property of Riemann integrals (when the underlying paths are continuously differentiable). Indeed, observe that for
    \[
    X_t^i X_t^j -X_0^i X_0^j = \int_0^t \int_0^r \dd X_u^i \dd X_r^j + \int_0^t \int_0^r \dd X_u^j \dd X_r^i 
    \]
    where we now interpret $\dd X_t = \frac{\dd}{\dd t}X_t\dd t.  $
\end{rem}

In later sections we will discuss in more detail time-extended paths $\hat{X}_t=(t,X_t)$ as first defined in Section \ref{sec:notation}.
Suppose there exists now a signature $\bX^{\leq \infty}$ above the path $X$. Then the extension to the signature over $t\mapsto \hat{X}_t$ can be done canonically. Indeed, since $t\mapsto t$ is smooth, the integrals 
\[
\int_0^t r \dd X_r, \quad \mathrm{and} \quad \int_0^t X_r \dd r, 
\]
can both be constructed as Riemann-Stieltjes integrals, since $X$ is assumed to be continuous. This argument however is not as simple in the setting where $X$ is rough, as the canonical construction of the rough paths is no longer clear. When we are starting out with a geometric rough path $\bX^{\leq p}$, then one can show that also the time augmented rough path $\hat{\bX}^{\leq p}$ will be geometric. This provides an important feature in the later proof of universality, where one typically depends on an integration by parts trick \cite{cuchiero2023universal}. As we are interested in non-geometric rough paths we need a deeper hypothesis on the structure of the rough path in order to prove later point separability in connection with universal approximation.

\subsection{Non-geometric rough paths}
In subsequent sections we will make use of non-geometric rough paths. In principle, these are rough paths given directly through Definition \ref{def:rough path}. However, we require some extra discussion on this point, as the lack of geometric nature of the rough paths at hand creates some challenges when working with them. When working with universal approximation of functionals, we need an injectivity argument, in the same spirit as used in \cite{cuchiero2023universal}. To this end, one needs to lift the path $t\mapsto X_t$ to the time-extended path $t\mapsto \hat{X}_t:=(t,X_t)$, and then construct the rough path above $\hat{X}$ only from the knowledge of the rough path  $\bX^{\leq p}$. In \cite{cuchiero2023universal} this can be done by invoking repeated integration by parts arguments. However, working with non-geometric rough paths deprives of us this possibility, and we will need some more details on the way to construct the general rough path $\hat{\bX}^{\leq p}$. 

\subsection{The regime $p<3$}
 When $p<3$, the rough path $\bX^{\leq 2} = (X^1,X^2)$, where $X^1_{s,t} = X_t-X_s$ for a path $X:[0,T]\rightarrow \RR^d$. In order to extend this to a rough path over $\hat{X}$ one can canonically define 
 \begin{equation}\label{eq:matrix}
     \hat{X}^2_{s,t} = \begin{pmatrix}
          \frac{(t-s)^2}{2} & \int_s^t r-s \dd X_r \\\int_s^t X_r-X_s \dd r & X^2_{s,t} 
     \end{pmatrix}
 \end{equation}
From this, the integral $\hat{X}^2$ takes values in $(\RR^{d+1})^{\otimes 2}$, and  the integral $\int_s^t r-s \dd X_r$ is well defined in the Young sense (since $t\mapsto t$ is smooth). Moreover, $\int_s^t X_r-X_s \dd r$ is defined in the Riemann-Stieltjes sense. Thus in the $p<3$ regime, the only new integrals that need to be defined in order to create $\hat{\bX}^{\leq 2}$ would be the three integrals in position $11,12$ and $21$ in the matrix \eqref{eq:matrix} which can simply be created by elementary means. 
\begin{rem}\label{rem:integration by parts for p<3}
    Note in particular that integration by parts holds for component $12$ and $21$ in the matrix \eqref{eq:matrix}  while it may not hold for any other component in $X^2$. Thus, $\hat{\bX}^{\leq 2}$ will not be geometric. 
\end{rem}

\subsection{The regime $p>3$}
Extending the picture to the $p>3$ regime requires more care. Now, $\bX^{\leq p}$ consists of $\lfloor p \rfloor>2$  levels of the rough path, and in order to construct the rough path over the time extension $\hat{X}$ one must make sure that all integrals are well defined. 

To this end, rough paths theory and the concept of controlled rough paths can be invoked for the construction of the various cross integrals appearing between $t$ and $X_t$. However, for this theory to be applied one needs some algebraic structure. In particular, as discovered in \cite{Gubinelli2010Ramification} and \cite{Hairer2015}, the concept of {\em branched rough paths} provides a suitably extended algebraic framework to still allow for constructions of the rough integration. 

For the purpose of this article where non-geometric rough paths are discussed broadly, we formulate an hypothesis to be called upon for some of the main results on functional approximation. 
But first, a definition.
\begin{defn}
    For any $X\in C^{p-var}([0,T],\RR^d)$ consider the time-augmented path $\hat{X}\in C^{p-var}([0,T],\RR^{d+1})$ defined by $\hat{X}_t=(t,X_t).$ We denote by
    \[\hat{V}^p(\triangle_T,T^{\lfloor p\rfloor}(\RR^{d+1}))\] the set of $p$-rough paths $\hat{\mathbf{X}}^{\leq p}$ over $\RR^{d+1}$ where the underlying path is time-augmented.
\end{defn}
\begin{hyp}\label{hyp:main}
    Let $\mathbf{X}^{\leq p}\in V^p(\triangle_T,T^{\lfloor p\rfloor}(\RR^d))$ be a $p$-rough path where $X\in C^{p-var}([0,T],\RR^d)$ is the underlying path. We assume:
    \begin{enumerate}
        \item There exists a $p$-rough path $\hat{\mathbf{X}}^{\leq p}\in \hat{V}^p(\triangle_T,T^{\lfloor p\rfloor}(\RR^{d+1}))$ where $\hat{X}$ is the underlying time augmentation of $X$.
        \item For any word $w$ of length less than or equal to $\lfloor p\rfloor$, then for all $n\in\NN$ there exists a linear functional $l_n\in T((\RR^{d+1}))^*$ with the property that
        \[\la l_n, \hat{\mathbf{X}}^{\leq\infty}_{0,T}\ra=\frac{1}{n!}\int_0^Ts^n\la w,\hat{\mathbf{X}}^{\leq p}_{0,s}\ra ds.\] Here $\hat{\mathbf{X}}^{\leq \infty}$ denotes the canonical extension of $\hat{\mathbf{X}}^{\leq p}$ to the full signature, guaranteed by Lyons extension theorem (see Thm. \ref{thm:lyons-extension}).
    \end{enumerate}
\end{hyp}
Throughout this article we assume the above hypothesis holds unless otherwise stated.

\begin{rem}
    Suppose the  hypothesis does not hold, i.e., we have a proper subset $\hat{\cS}^p\subsetneq \hat{V}^p(\triangle_T,T^{\lfloor p\rfloor}(\RR^{d+1}))$ which is the set of $p$-rough paths satisfying Hypothesis \ref{hyp}. For any $p\geq 1$, $\hat{\cS}^p$ is nonempty. Moreover, $\hat{\cS}^p$ is closed with respect to the $p$-variation metric: Indeed, one can show that for every weighted word $\pi\in T(\RR^{d+1})$, the linear functional $\la \pi,\cdot\rangle$ restricted to the Fock space is moreover bounded. In other words, the functional $\la \pi,\cdot\ra$ is an element of the topological dual of $\cF(\RR^{d+1})$. Thus, using continuity of the signature map, and the functionals $\la l_n,\cdot\ra,$ $\la w,\cdot\ra$, we get that every convergent sequence in $\hat{\cS}^p$, converges to a point in $\hat{\cS}^p$. Hence, if $\mathcal{K}\subset \hat{V}^p(\triangle_T,T^{\lfloor p\rfloor}(\RR^{d+1}) $ is compact, so is the intersection $\cK\cap \hat{\cS}^p$. Since compactness is an important ingredient in Stone-Weierstra{\ss} theorem, this remark ensures that our main results hold for this subset. Thus if the hypothesis is not true for all time-augmented paths, we consider this closed subset.
\end{rem}

\section{Functional approximation with signatures}\label{sec:universal}
With the goal of approximating complex pricing functionals from financial markets, we recall here some basic properties of universality of the signature. In addition, we provide a new statement of the universal approximation property for functionals acting on general rough paths (not restricted to geometric rough paths). In this section we will frequently use the set \begin{align*}&\hat{V}^p([0,T],T^{\lfloor p\rfloor}(\RR^{d+1}))\\&:=\left\{\hat{\mathbf{X}}^{\leq p}_{0,\cdot}:[0,T]\rightarrow T^{\lfloor p\rfloor}(\RR^{d+1})\mid \hat{\mathbf{X}}^{\leq p}\in \hat{V}^p(\triangle_T,T^{\lfloor p\rfloor}(\RR^{d+1}))\right\}\end{align*} of time-augmented $p$-rough paths where the first time component is evaluated at zero. Similarly, we define
\[\hat{V}^p([0,T],G^{\lfloor p\rfloor}(\RR^{d+1}))\] to be the set of geometric $p$-rough paths where the first time slot is evaluated at zero.

\subsection{Universal approximation for geometric rough paths}
The functional approximation setup outlined in \cite{Arribasetal20}, based on ideas also formulated in \cite{arribas2018derivativespricingusingsignature}, is strongly based on the universal approximation property of the signature. This property is a consequence of the Stone-Weierstrass theorem, using the fact that the linear span of signatures of geometric rough paths forms an algebra, and that when lifting the underlying path $X$ to its time extension 
$\hat{X}_t=(t,X_t)$, the signature $\hat{\bX}^{\leq\infty}$  uniquely determines the path, and thus separates points. The universal approximation theorem relied upon in \cite{arribas2018derivativespricingusingsignature,Arribasetal20} is based on Lipschitz-paths. The following universal approximation theorem is based on the recently proposed extension by Cuchiero et.al. in \cite{cuchiero2023universal} to the setting of $G^{\lfloor p \rfloor}(\RR^d)$-valued paths.

\begin{prop}[Universal approximation]\label{prop:universal approximation}
   For $p\geq 1$ let $\cK\subset \hat{V}^p([0,T];G^{\lfloor p \rfloor}(\RR^{d+1}))$ be a compact subset which is finite in the $p$-variation  norm. Suppose $F$ is a functional on $  \hat{V}^p([0,T];G^{\lfloor p \rfloor}(\RR^{d+1}))$. Let    $\hat{\bX}^{\leq \lfloor p\rfloor }\in \cK$ be a $p$-rough path  defined  according to Definition \ref{def:rough path} over the extended path $\hat{X}_t=(t,X_t)$ for some $X:[0,T]\rightarrow\RR^d$. Then for each $\epsilon>0$ there exists a linear functional $\pi\in T((\RR^{d+1}))^*$ such that  
   \begin{equation*}
      \|F-\la \pi,S(\cdot)_{0,T}\ra\|_{\infty,\cK}= \sup_{\hat{\bX}^{\leq p}\in \cK} |F(\hat{\bX}^{\leq \lfloor p\rfloor})-\la \pi,\hat{\bX}^{\leq \infty }_{0,T}\ra |<\epsilon, 
   \end{equation*}
   where we use the supremum norm on $C(\cK)$, $\|\cdot\|_{\infty,\cK}$, and $S(\cdot)_{s,t}:\cK\rightarrow T((\RR^{d+1}))$ is the signature map $\hat{\bX}^{\leq p}\mapsto  \hat{\bX}^{\leq \infty}_{s,t}$ evaluated at $(s,t)\in\triangle_T$. 
\end{prop}
The universal approximation applied to $G^{\lfloor p\rfloor}(\RR^{d+1})$ valued paths allows us, in particular, to consider functionals of truly rough signals $\hat{X}$, that only have finite $p\geq 1$ variation, as long as we have constructed the rough path of $\hat{\bX}^{\leq p}$ over $\hat{X}$. For instance, we can consider functionals of paths of the Brownian motion. However, in its current form, it is only formulated for Stratonovich lifts of the iterated integral, excluding the more natural choice of rough paths lift for financial applications, namely the It\^o lift. 
One might try to work around this problem by identifying the It\^o-Stratonovich correction term and implementing this in the functional and signature. However, under the working hypothesis that will be used in the remainder of the text, we can circumvent this challenge completely.

\subsection{General universal approximation on rough paths}

The universal approximation theorem in Proposition \ref{prop:universal approximation} heavily relies upon the geometric structure of the rough paths, on which the functionals $F$ act on. The reason is that when applying the Stone-Weierstrass theorem to check for denseness of the linear span of signature terms in the space $\hat{V}^p([0,T];G^{\lfloor p\rfloor}(\RR^{d+1}))$, one relies upon being able to multiply two signature terms $\la \pi_1,\hat{\bX}^{\leq \infty}\ra $ and $ \la \pi_2,\hat{\bX}^{\leq \infty}\ra$  to obtain a new signature term $\la \pi ,\hat{\bX}^{\leq \infty}\ra$. When $
\hat{\bX}^{\leq   p} $ is geometric (i.e. takes values in $G^{\lfloor p \rfloor}(\RR^{d+1})$) it follows from Lemma \ref{lemma:shuffle} that this holds with $\pi = \pi_1\shuffle \pi_2 $.  However, this is a restrictive class of signatures; in the semi-martingale setting it corresponds to Stratonovich lifts of the rough path. In financial applications, one typically work with functions that acts on non-geometric rough paths, such that It\^o lifts of semi-martingales. Thus extending the universal approximation property to any rough path provides practical consequences for several applications. 

It turns out that a simple mixing of ideas from the Stone-Weierstrass theorem over polynomials, with the classical signature universal approximation allows one to obtain such a new generalized universal approximation. 
To this end we will need a 
so-called separation of points property of the sub-algebra of continuous functionals that we will consider as a the basis for functional approximation. We therefore recall the following technical lemma from \cite[Cor. 4.24]{Brezis2011}. 
\begin{lem}\label{lem:technical}
Let $\Omega\subset \RR^d$ be open, and suppose $u\in L^1_{loc}(\Omega)$ is such that
\[
\int_\Omega u(z) f(z)\dd z = 0,\quad \forall f\in C^\infty_c(\Omega). 
\]
Then $u\equiv 0$ a.e on $\Omega$.
\end{lem}
With this lemma at hand we are now ready to prove a generalized version of polynomial universal approximation over rough paths with values in the tensor algebra $T((\RR^{{d+1}}))$. 

\begin{thm}[Generalized universal approximation]\label{thm:universal general}
    Let $F$ be a continuous functional on a compact set  $\mathcal K\subset\hat{V}^p([0,T];T^{\lfloor p\rfloor}(\RR^{d+1}))$. Then for any $\epsilon>0$ there exists finite set $\cN\subset \NN^n$, and a polynomial $\bar{f}_{\cN}:\RR^n\rightarrow \RR$ given by 
    \begin{equation}\label{eq:multivariate polynomial}
        \bar{f}_{\cN}(x)= \sum_{m\in \cN} \alpha_m x^m, 
    \end{equation}
    and a sequence of linear operators $\{\pi_i\}_{i=1}^n \subset T(\RR^{d+1})$ with the property that for all $\hat{\bX}_{0,\cdot}^{\leq p}\in \mathcal K$ then 
    \begin{equation*}
        \left|F(\hat{\bX}_{0,\cdot}^{\leq p})_T-\bar{f}_{\cN}\left(\la \pi_1,\hat{\bX}^{\leq \infty}_{0,T}\ra,\ldots,\la \pi_n,\hat{\bX}^{\leq \infty}_{0,T}\ra\right)\right|<\epsilon. 
    \end{equation*}
\end{thm}
\begin{proof}
   Let $\cE^p:= \hat{V}^p([0,T];T^{\lfloor p\rfloor} (\RR^{d+1}))$. 
Define for $n\in \NN$ and $p\geq 1$,
\begin{equation}
    \cA_{n,p} = \mathrm{span}\{ \hat{\bX}_{0,\cdot}^{\leq p } \mapsto  \prod_{i=1}^n\la \pi_i,\hat{\bX}_{0,T}^{\leq \infty}\ra^{m_i} \, |\,\,  m\in \NN^n,\, \{\pi_i\}_{i=1}^n\subset T((\RR^{d+1})^*),\, \hat{\bX}_{0,\cdot}^{\leq p} \in \cE^p  \}.
\end{equation}
Furthermore, let $\cA_p= \bigcup_{n=1}^\infty \cA_{n,p}$. Clearly, $\cA_p\subset C(\cE^p;\RR)$. We denote by $\cA_p|_\cK=\{P|_\cK\in C(\cK)\mid P\in \cA_p\}$ the set $\cA_p$ where we restrict the domain to $\cK$.  With the goal of applying Stone-Weierstrass theorem to prove denseness of  $\cA_p|_\cK$ in $C(\cK)$ we check that the following holds: 
\begin{itemize}[leftmargin=*]
    \item[i)] The set $\cA_p$ forms a sub-algebra.
    Indeed; addition holds. Furthermore, for two elements $A\in\cA_{n,p},B\in\cA_{k,p}$ given by $A(\hat{\bX}^{\leq p}_{0,\cdot})=\prod_{i=1}^n\la \pi_i,\hat{\bX}_{0,T}^{\leq \infty}\ra^{m_i}\in \RR$ and $B(\hat{\bX}^{\leq p}_{0,\cdot})=\prod_{i=1}^k\la \pi'_i,\hat{\bX}^{\leq \infty}\ra^{m_i'}\in \RR$ then we can choose $\{\tilde{\pi_i}\}$ such that 
    \begin{align*}
    (AB)(\hat{\bX}^{\leq p}_{0,\cdot})&=A(\hat{\bX}^{\leq p}_{0,\cdot})B(\hat{\bX}^{\leq p}_{0,\cdot})= \prod_{i=1}^n\prod_{j=1}^k\la \pi_i,\hat{\bX}_{0,T}^{\leq \infty}\ra^{m_i}\la \pi'_i,\hat{\bX}_{0,T}^{\leq \infty}\ra^{m'_i}\\&= \prod_{i=1}^{n+k}\la \tilde{\pi}_i,\hat{\bX}_{0,T}^{\leq \infty}\ra^{\tilde{m}_i}=:C(\hat{\bX}^{\leq p}_{0,\cdot}).
    \end{align*} 
    where  $\tilde{\pi}_i = \pi_i$ for $i\in \{1,...,n\}$ and $\tilde{\pi}_i=\pi'_{i-n}$ for $i=n+1,...,n+k$, and $\tilde{m_i}$ is the concatenation of $m_i$ and $m_i'$. Thus $AB=C\in \cA_{k+n,p}$.
    \item[ii)] $\cA_p$ separates points. 
    Let $\hat{\bX}_{0,\cdot}^{\leq p},\hat{\bY}_{0,\cdot}^{\leq p}\in \cE^p$.
    We need to show that for $\hat{\bX}_{0,\cdot}^{\leq p}\neq\hat{\bY}_{0,\cdot}^{\leq p}$, there exists a $P\in \cA_p$ such that $P(\hat{\bX}_{0,\cdot}^{\leq p})\neq P(\hat{\bY}_{0,\cdot}^{\leq p})$. Since  \[(1,\hat{X}_{0,\cdot}^1,\dots,\hat{X}_{0,\cdot}^{\lfloor p\rfloor})=\hat{\mathbf{X}}_{0,\cdot}^{\leq p}\neq\hat{\mathbf{Y}}_{0,\cdot}^{\leq p}=(1,\hat{Y}_{0,\cdot}^1,\dots,\hat{Y}_{0,\cdot}^{\lfloor p\rfloor}),\] there exists a $j\in \{1,\dots, \lfloor p\rfloor\}$, a word $w\in\mathcal{W}^j$ of length $j$, and an $s\in[0,T]$, such that \[\la w, \hat{\mathbf{X}}^{\leq \infty}_{0,s}\ra=\la w,\hat{\mathbf{X}}^{\leq p}_{0,s}\ra\neq \la w,\hat{\mathbf{Y}}^{\leq p}_{0,s}\ra=\la w, \hat{\mathbf{Y}}^{\leq\infty}_{0,s}\ra.\]

Now, assume for contradiction that $\mathcal{A}_p$ does not separate points, i.e., $P(\hat{\mathbf{X}}^{\leq p}_{0,\cdot})=P(\hat{\mathbf{Y}}^{\leq p}_{0,\cdot})$ for all $P\in\mathcal{A}_p$. Then by Hypothesis \ref{hyp:main}, for any $n\in\NN$ let the word $l_n$ be such that
\[\la l_n, \hat{\mathbf{X}}^{\leq\infty}_{0,T}\ra=c\int_0^Ts^n\la w, \hat{\mathbf{X}}^{\leq p}_{0,s}\ra ds,\] and define $P_{l_n}\in \mathcal{A}_p$ by
\[P_{l_n}(\cdot)=\la l_n, S(\cdot)_{0,T}\ra.\] Since $P_{l_n}(\hat{\mathbf{X}}^{\leq p}_{0,\cdot})=P_{l_n}(\hat{\mathbf{Y}}^{\leq p}_{0,\cdot})$ we further obtain the equation:

\[c\int_0^Ts^n(\la w,\hat{\mathbf{X}}^{\leq p}_{0,s}\ra-\la w, \hat{\mathbf{Y}}^{\leq p}_{0,s}\ra)ds=0,\quad \text{ for all }n\in\NN.\] Then since polynomials are dense in $C([0,T];\RR)$ it follows by Lemma \ref{lem:technical} that\[\la w,\hat{\mathbf{X}}^{\leq p}_{0,s}\ra=\la w, \hat{\mathbf{Y}}^{\leq p}_{0,s}\ra,\quad \text{ for all } s\in[0,T],\] which yields a contradiction.    
       As a consequence it implies that $\cA_p$ separate points.  
    \item[iii)] The constant function $1$ is in $\cA_p$, since by choosing $\pi=\emptyset\in T((\RR^{d+1}))^*$ we get that $\la \emptyset, \hat{\bX}^{\leq \infty}\ra=1 $ and therefore in particular $\la \emptyset,\hat{\bX}^{\leq p}_{0,\cdot}\ra =1 $.
\end{itemize} 
From these three properties, it follows by the Stone-Weierstrass theorem that for any compact  $\cK\subset\cE_p$, then $\cA_p|_\cK$, the restriction of $\cA_p$ on $\cK$, is dense in $C(\cK;\RR)$, which concludes the proof.  
\end{proof}

\subsection{Signature associated to price paths}

As discussed in the introduction, even the most complex financial derivatives typically have a simple functional structure. By this we mean that the path dependent nature of the functional either comes through an averaging over the price path (like in Asian style derivatives), or products of price paths (quanto-style options), in addition to basket of different assets etc. with these structures. 
It is therefore natural to assume that for this purpose, given a price path $X$, and the extended signature $\hat{\bX}^{\leq \infty}$, the payoff functional $F(\hat{\bX}^{\leq p}_{0,\cdot})$ can be written as a function $f:\RR^n\rightarrow \RR$, and $n$ different weighted words $\pi_i$, such that 
\[
 F(\hat{\bX}^{\leq p}_{0,\cdot})=f(\la \pi_1,\hat{\bX}_{0,T}^{\leq \infty})\ra,\ldots,\la \pi_n,\hat{\bX}_{0,T}^{\leq \infty}\ra  ). 
\]
This will therefore be the main working hypothesis of the subsequent sections, and we will illustrate several numerical and analytic advantages of using this specific structure. We will also give examples to show exactly how this hypothesis applies to various exotic derivatives. As a first simple example, we have the following:
\begin{ex}
\label{ex:asian-quanto}
Asian options are contracts that pay the holder an amount of money according to the average price over a period of time. If $X$ is the price process, the holder receives $f(\int_0^TX_sds)$ at time $T$. We let the signature of $X$ be the canonical choice determined by the model. In classical models of Asian options, the price process $X$ is a semi-martingale, hence the canonical choice in this case would be the Itô signature. By assuming Hypothesis \ref{hyp}, consider the time-enhanced price path $\hat{X}$ defined by $\hat{X}_t=(t,X_t)$, we see (recall Example \ref{ex:word-functional-identification} with $d=2$ and $\mathcal{A}=\{\mathbf{1},\mathbf{2}\}$) that $\int_0^TX_tdt=\int_0^T\int_0^tdX_rdt+X_0 T=\langle w,\hat{\mathbf{X}}^{\leq\infty}_{0,T}\rangle$ for 
$w=\mathbf{2}\mathbf{1}+X_0\mathbf{1}$. Another example is a spread option between two assets with price dynamics $X_1$ and $X_2$, respectively, paying the holder $\max(X_{1,T}-c X_{2,T},0)$ at time $T$, for a conversion constant $c$ (here, $c$ may convert the currency of the second asset into the currency of the first, say). With $\hat{X}=(t,X_1,X_2)$, we can (still following the notation in Example \ref{ex:word-functional-identification}, now with $d=3$) express the payoff as 
$$
f(\langle w_1,\hat{\mathbf{X}}_{0,T}^{\leq\infty}\rangle,\langle w_2,\hat{\mathbf{X}}_{0,T}^{\leq\infty}\rangle)=\max(\langle w_1,\hat{\mathbf{X}}_{0,T}^{\leq\infty}\rangle-c\langle w_2,\hat{\mathbf{X}}_{0,T}^{\leq\infty}\rangle,0)
$$
with $w_1=\mathbf{2}+X_{1,0}\emptyset$ and $w_2=\mathbf{3}+X_{2,0}\emptyset$,
or, more simple,
$$
f(\langle\pi,\hat{\mathbf{X}}_{0,T}^{\leq\infty}\rangle)=\max(\langle \pi,\hat{\mathbf{X}}_{0,T}^{\leq\infty}\rangle,0)
$$
with the weighted word $\pi=w_1-c w_2$. Yet another example from energy finance is so-called quanto-options (see e.g. \cite{BLM}), where the holder receives a payment at exercise time $T$ according to a product of two payoffs on the average of the spot energy price and temperature, say. Denoting $X_1$ the energy spot price, $X_2$ the temperature process, and $\hat{X}=(t,X_{1},X_{2})$, we have a payoff
$$
f(\langle w_1,\hat{\mathbf{X}}_{0,T}^{\leq\infty}\rangle,\langle w_2,\hat{\mathbf{X}}_{0,T}^{\leq\infty}\rangle)=g(\langle w_1,\hat{\mathbf{X}}_{0,T}^{\leq\infty}\rangle)h(\langle w_2,\hat{\mathbf{X}}_{0,T}^{\leq\infty}\rangle)
$$
where $w_1=\mathbf 2\mathbf 1+X_{1,0}\mathbf{1}$ and $w_2=\mathbf 3\mathbf 1+X_{2,0}\emptyset$, and $g,h$ are the payoff functions written on the average of the spot energy price and temperature, resp. 
\end{ex}

Let us precise the hypothesis we work under in the remainder of this paper. 

\begin{hyp}\label{hyp}
 Let $p\geq 1$ and  assume that for a given continuous functional $F$ acting on $\hat{V}^p([0,T];T^{\lfloor p\rfloor}(\RR^{d+1}))$, there exists a collection of linear operators  $\{\pi_i\}_{i=1}^n \in T(\RR^{d+1}) $, and a function $f:\RR^n \rightarrow \RR$ such that for any $\hat{\bX}^{\leq p}_{0,\cdot}\in \hat{V}^p([0,T];T^{\lfloor p\rfloor}(\RR^{d+1}))$
    \begin{equation}\label{eq:hyp representation}
    F(\hat{\bX}^{\leq p}_{0,\cdot})=f(\la \pi_1,\hat{\bX}_{0,T}^{\leq \infty})\ra,\ldots,\la \pi_n,\hat{\bX}_{0,T}^{\leq \infty}\ra  ). 
\end{equation}
\end{hyp}

The main advantage of invoking Hypothesis \ref{hyp} is that polynomials are dense in the space of continuous functions on a compact. Thus, if Hypothesis \ref{hyp} holds for a continuous function $f$, then we may approximate the functional $\hat{\bX}^{\leq p} \mapsto F(\hat{\bX}^{\leq p})$  by a (multivariate) polynomial in a finite number of signature coefficients. More precisely, if Hypothesis \ref{hyp} holds,  then for all $\epsilon>0$, there exists a polynomial $\bar{f}$ such that $\|f-\bar{f}\|<\epsilon$ and $\bar{f}$ is of the form
\[
\bar{f}(x)= \sum_{m\in \NN^{n},|m|\leq N} \alpha_m x^m
\]
where for a multi-index $m\in \NN^n$ and $x\in \RR^n$, we write $x^m=x_1^{m_1}\dots x_n^{m_n}$, and $|m|=m_1+\dots+m_n$. The coefficients $\alpha_m$ are real numbers labeled by the multi-index $m\in \NN^n$. 
This leads to a simplified version of the universal approximation theorem, being a consequence of the Stone-Weierstrass theorem and Hypothesis \ref{hyp}.

\begin{thm}\label{thm:universal under hyp}
    Let $F$ be a continuous functional on $\hat{V}^p([0,T];T^{\lfloor p\rfloor}(\RR^{d+1}))$, and suppose Hypothesis \ref{hyp} holds with a sequence $\{\pi_i\}_{i=1}^n$ and a \emph{continuous} function $f$. Then for any $\epsilon>0$ there exists a finite set $\cN\subset \NN^n$, a compact set $\cK_\epsilon\subset \RR^n$, and a polynomial $\bar{f}_{\cN}:\RR^n\rightarrow \RR$ given by 
    \begin{equation*}
        \bar{f}_{\cN}(x)= \sum_{m\in \cN} \alpha_m x^m, 
    \end{equation*}
    with the property that 
    \begin{equation}
    \label{property-uat}
        \left|F(\hat{\bX}_{0,\cdot}^{\leq p})-\bar{f}_{\cN}\left(\la \pi_1,\hat{\bX}_{0,T}^{\leq \infty}\ra,\ldots,\la \pi_n,\hat{\bX}_{0,T}^{\leq \infty}\ra\right)\right|<\epsilon, 
    \end{equation}
    for all $\hat{\bX}^{\leq p}_{0,\cdot}\in \hat{V}^p([0,T];T^{\lfloor p\rfloor}(\RR^{d+1})) $ such that $\left(\la \pi_1,\hat{\bX}_{0,T}^{\leq \infty}\ra,\ldots,\la \pi_n,\hat{\bX}_{0,T}^{\leq \infty}\ra\right) \in \cK_\epsilon $. 
    Furthermore, let $K>0$ be the constant from Theorem \ref{thm:lyons-extension} such that for $\pi_i\in T((\RR^{d+1})^*)$ of the form $\pi_i = \sum_{j=1}^N \kappa_{ij} e_{w_{ij}}$ with $\{\kappa_{ij}\}_{j=1}^N\subset \RR$, 
   \[
   |\la \pi_i,\hat{\bX}_{0,T}^{\leq \infty} \ra |  \leq  \sum_{j=1}^N |\kappa_{ij}|\frac{K^{|w_{ij}|}}{(|w_{ij}|/p)!}, 
   \] 
   where $|w_i|=|e_{w_i}|$. 
  Suppose the coefficients $\{\alpha_m\}_{m\in \cN}$ satisfies for some $C>0$ 
    \begin{equation}\label{eq:a conditions}
    |\alpha_m|\leq \frac{C^{|m|}}{m!}  \prod_{i=1}^n \left(\sum_{j=1}^{N_i} |\kappa_{ij}|\frac{(|w_{ij}|/p)!}{K^{|w_{ij}|}}\right)^{m_i}, 
     \end{equation}
    where $m!=m_1!m_2!\ldots m_n!$. 
    Then we have that 
    \begin{equation}\label{eq:approximation bound general}
          \left|F(\hat{\bX}^{\leq p}_{0,\cdot})-\bar{f}_{\cN}\left(\la \pi_1,\hat{\bX}_{0,T}^{\leq \infty}\ra,\ldots,\la \pi_n,\hat{\bX}_{0,T}^{\leq \infty}\ra\right)\right| \leq \frac{\exp(C)C^{|\cN|+1}}{(n-1)!(|\cN|-n-1)!}.
    \end{equation}
\end{thm}
\begin{proof}
    This is a simple consequence of the Stone-Weierstrass approximation theorem for continuous functions, using that there always exists a compact subset $\cK_\epsilon\subset \RR^n$ such that 
    \[
    \left(\la \pi_1,\hat{\bX}_{0,T}^{\leq \infty}\ra,\ldots,\la \pi_n,\hat{\bX}_{0,T}^{\leq \infty}\ra\right)\in \cK_\epsilon.
    \]
    Restricting the domain of $f$ to $\cK_\epsilon$, we are done showing \eqref{property-uat}. For the convergence rate \eqref{eq:approximation bound general}, define the remainder term
    \[
    R := \sum_{m\in \NN^n\setminus \cN} \alpha_m \left( \la \pi_1,\hat{\bX}_{0,T}^{\leq \infty}\ra,\ldots,\la \pi_n,\hat{\bX}_{0,T}^{\leq \infty}\ra \right)^m  
    \]
    Invoking the bound on the signature decay and the assumption on $|\alpha_m|$, we see that 
\begin{equation*}
    |R|\leq  \sum_{m\in \NN^n\setminus \cN} |a_m| \prod_{i=1}^n \left(\sum_{j=1}^N |\kappa_{ij}|\frac{K^{|w_{ij}|}}{(|w_{ij}|/p)!}\right)^{m_i}\leq  \sum_{m\in \NN^d\setminus \cN} \frac{C^{|m|}}{m!}.  
\end{equation*}
The right hand side of this inequality corresponds to  the remainder term of a multivariate Taylor approximation of the function $\exp(C\prod_{i=1}^nx_i)$ around $0$ up to order $m\in\mathcal N$. Thus from the multivariate Taylor theorem, it follows that 
\[
|R| \leq  \frac{\exp(C)C^{|\cN|+1}}{|\cN|!}\sum_{|m|=|\cN|+1}1
\]
By an elementary combinatorial argument (using the so-called "stars and bars"-argument) we see that $\sum_{|m|=|\cN|+1}1=\binom{|\cN|}{n-1}$. This concludes the proof.

\end{proof}

\begin{rem}
 While Hypothesis \ref{hyp} is certainly limiting the class of functionals $F$ that we can analyze, universal approximation becomes easier. In addition, the assumption that $F$ only acts on compact subsets of $p$-variation paths with values in the space of geometric rough paths, i.e. $\hat{V}^p([0,T];G^{\lfloor p\rfloor}(\RR^{d+1}))$,  is dropped, allowing for an easier verification of universality. For later probabilistic arguments related to financial prices as expected functionals, this point will simplify computations and discussions.  

Furthermore, the classical assumption that the approximation holds over compact subsets $\cK\subset \hat{V}^p([0,T];T^{\lfloor p\rfloor}(\RR^{d+1})) $, as seen in Theorem \ref{thm:universal general}, the compactness statement in Theorem \ref{thm:universal under hyp} significantly simplifies this. Indeed, describing compact subsets of $\hat{V}^p([0,T];T^{\lfloor p\rfloor}(\RR^{d+1}))$ can be a challenging task, as illustrated in e.g. \cite{gulgowski24}. In Theorem \ref{thm:universal under hyp} one essentially only need to choose a bound $M>0$, and one can consider any $\hat{\bX}^{\leq p}_{0,\cdot} \in\hat{V}^p([0,T];T^{\lfloor p\rfloor}(\RR^{d+1}))$ such that 
\[
|\la \pi_i, \hat{\bX}_{0,T}^{\leq \infty }\ra  |\leq M,\quad \forall \,\,i=1,\ldots,n. 
\]
The compact subsets of an infinite dimensional space is therefore replaced by a (something that may be interpreted as)  bounded subsets. This can also make probabilistic statements easier, as will be illustrated in subsequent sections.  
\end{rem}

\begin{rem}
    It is important to note that in Theorem $\ref{thm:universal under hyp}$ the statement allows for any functional acting on the space of $p$-variation paths with values in the truncated tensor algebra. This is a significant difference with the classical universal approximation theorem for signatures stated in Proposition \ref{prop:universal approximation}, as the space of geometric rough paths limits the possible structure of the $F$ under consideration. In particular, a canonical asset pricing model would be constructed from semi-martingales and It\^o processes. To preserve a martingale property of the derivative prices, one then use the It\^o integral for computing derivatives prices, an integration choice which in the sense of signatures is not geometric. In contrast, under Hypothesis \ref{hyp} one can easily work with functionals that structurally contain It\^o integration, and still obtain a direct and descriptive approximation of the functional in terms of the signature associated to the price path. 
\end{rem}

\begin{rem}
    Computationally, Hypothesis \ref{hyp}, given a specific $F$, one  only requires the computation of $n$ terms from the signature, and not the full signature, and with these $n$ terms one can achieve as high accuracy as desired for functional approximation. This is in stark contrast to the much more general Universal approximation theorem in Proposition \ref{prop:universal approximation} and Theorem \ref{thm:universal general}, where the accuracy of the approximation is dictated by number of signature terms included. Invoking Hypothesis \ref{hyp} therefore has the potential to reduce computational time significantly. 
\end{rem}
\begin{rem}
    The condition assumed on the coefficients $\{\alpha_m\}$ in \eqref{eq:a conditions} yields the bounds in \eqref{eq:approximation bound general}. Different assumptions on $\{\alpha_m\}$ will yield different convergence rates. While the condition in \eqref{eq:a conditions} is seemingly abstract, it can be verified to be weaker than the conditions satisfied by the coefficients in a Taylor expansion. On the other hand, the condition is not satisfied by a much "slower" convergent polynomial series, such as the Bernoulli polynomials. A more clear illustration of this condition will be given by the subsequent examples. 
\end{rem}

A restriction of the functional approximation in Theorem \ref{thm:universal under hyp} to the case of functionals on geometric rough paths can readily be seen as a special case of the classical universal approximation theorem presented in Proposition \ref{prop:universal approximation}. Indeed, by the next lemma we obtain a corollary that coincides with Proposition \ref{prop:universal approximation}.

\begin{lem}\label{lem:cont}
    For $n\in\NN$ and weighted words $\pi_1,\dots,\pi_n\in \RR\la \{1,\dots,d+1\}\ra$, define a function $g:\hat{V}^{p}([0,T],T^{\lfloor p\rfloor}(\RR^{d+1}))\rightarrow \RR^n$ by \[g(\hat{\mathbf{X}}_{0,\cdot}^{\leq p})=\left(\la\pi_1,\hat{\mathbf{X}}^{\leq\infty}_{0,T}\ra,\dots,\la\pi_n,\hat{\mathbf{X}}^{\leq\infty}_{0,T}\ra\right),\]where the toppology on $\hat{V}^{p}([0,T],T^{\lfloor p\rfloor}(\RR^{d+1}))$ is induced by the $p$-variation metric (see \ref{def:p-var}).  Then $g$ is continuous.
\end{lem}
\begin{proof}
    Since all norms on $\RR^n$ are equivalent, we will in this proof consider the $\infty$-norm. For any $\epsilon>0$ we can choose a $\delta>0$ such that the Mittag-Leffler function satisfies $E_{1/p,0}(\delta)\leq \epsilon \frac{\beta(p)}{\|\pi_{k'}\|_{\cF(\RR^{d+1})}}$, where $k'\in\{1,\dots, n\}$ satisfies $\|\pi_{k'}\|=\max_{1\leq k\leq n}\|\pi_k\|_{\cF(\RR^{d+1})}$. Then for any elements in  $\hat{V}^{p}([0,T],T^{\lfloor p\rfloor}(\RR^{d+1}))$ satisfying $\|\hat{\mathbf{X}}_{0,\cdot}^{\leq p}-\hat{\mathbf{Y}}_{0,\cdot}^{\leq p}\|_{p,[0,T]}<\delta,$ we have that
    \[\|g(\hat{\mathbf{X}}_{0,\cdot}^{\leq p})-g(\hat{\mathbf{Y}}_{0,\cdot}^{\leq p})\|_\infty=\max_{1\leq k\leq n}\left|\la \pi_k,\hat{\mathbf{X}}_{0,T}^{\leq \infty}-\hat{\mathbf{Y}}_{0,T}^{\leq \infty}\ra\right|<\epsilon.\]
\end{proof}

\begin{cor}
      Let $F$ be a continuous functional on the space of $p$-variation (extended) geometric rough paths,  $\hat{V}^p([0,T];G^{\lfloor p\rfloor}(\RR^{d+1}))$, and suppose Hypothesis \ref{hyp} holds with a sequence $\{\pi_i\}_{i=1}^n$ and a \emph{continuous} function $f$. Then for any $\epsilon>0$ there exists finite set $\cN\subset \NN^n$ and a compact $\cK\subset \hat{V}^p([0,T];G^{\lfloor p\rfloor}(\RR^{d+1}))$, such that 
    \begin{equation*}
       \left|F(\hat{\bX}^{\leq p}_{0,\cdot})- \sum_{m\in \cN} \alpha_m \la \phi_m ,\hat{\bX}_{0,T}^{\leq \infty}\ra \right|<\epsilon, \quad \forall \hat{\bX}^{\leq p}_{0,\cdot}\in \cK.
    \end{equation*}
\end{cor}
\begin{proof}
    Let $\cK\subset\hat{V}^p([0,T];G^{\lfloor p\rfloor}(\RR^{d+1}))$ be compact (choose a finite union of singletons if you must) and consider the function $g:\hat{V}^p([0,T];G^{\lfloor p\rfloor}(\RR^{d+1}))\rightarrow \RR^n$ given by $g(\hat{\mathbf{X}}^{\leq p}_{0,\cdot})=\left(\la \pi_1,\hat{\mathbf{X}}^{\leq \infty}_{0,T}\ra,\dots,\la \pi_n,\hat{\mathbf{X}}^{\leq \infty}_{0,T}\ra\right)$. By Lemma \ref{lem:cont}, we have that $\mathcal{K}_\epsilon:=g(\mathcal{K})$ is compact. Thus $f|_{\mathcal{K}_\epsilon}\in C(\mathcal{K}_\epsilon)$ is a continuous function on a compact set and we can find a polynomial approximating $f|_{\cK_\epsilon}$ by the classical Stone-Weierstass theorem.
    Since now $\hat{\bX}^{\leq \infty}$ is a geometric rough path, it follows from Lemma \ref{lemma:shuffle} that for $m\in \NN^n$ there exists a linear functional  $\phi_m\in T((\RR^{d+1})^*)$ such that 
    \[
    \left(\la \pi_1,\hat{\bX}_{0,T}^{\leq \infty}\ra,\ldots,\la \pi_n,\hat{\bX}_{0,T}^{\leq \infty}\ra\right)^m = \la  \phi_m, \hat{\bX}_{0,T}^{\leq \infty}\ra. 
    \]
    Thus, inserting this into the polynomial, the result follows. 
\end{proof}
\begin{rem}\label{rem:computational efficiency}

    Note that the linear functionals $\phi_m$ quickly become very large sums of words, even when the $\pi_i$'s consist of elementary words. As an example, for some single letter $a\in \{1,\ldots, d\}$, consider the product $\la a,\hat{\bX}_{0,T}^{\leq \infty}\ra^k$ for some potentially large $k\in \NN$. Then doing the $k$th power of the shuffle product of $a$, yields the word $a\ldots a$ ($a$ repeated $k$ times), and we get the weight $k!$ in front, i.e., 
    \[
    \la a,\hat{\bX}_{0,T}^{\leq \infty}\ra^k=k!\la a\ldots a,\hat{\bX}_{0,T}^{\leq \infty} \ra.  
    \]
    See, e.g., \cite{BOWMAN200243} and the references therein for a longer exposition of the shuffle product and algebras. 
   It becomes quickly expensive to compute higher order signature terms. However, as long as Hypothesis \ref{hyp} is in place, signature computations can be made much more efficient if what one really needs is only to compute the power of the number $\la a,\hat{\bX}_{0,T}^{\leq \infty}\ra $. When computing expected values, this is often the situation. 
    
\end{rem}

\subsection{Examples of functions}

We will in this Subsection consider a few examples of functions that can be approximated, and investigate their convergence properties.  As already discussed, most examples of financial payoff functionals only considers the simpler case when Hypothesis \ref{hyp} holds. That is, for each specific payoff functional $F(\hat{\bX}^{\leq p}_{0,\cdot})$  there exists  a finite number of linear operators $\pi_i\in T(\RR^{d+1})$ for $i=1,\ldots,n$ such that 
\[
F(\hat{\bX}_{0,\cdot}^{\leq p}) = f(\la \pi_1, \hat{\bX}^{\leq \infty}_{0,T}\ra,\ldots, \la \pi_n, \hat{\bX}^{\leq \infty}_{0,T}\ra).
\] 
From both the Universal approximation theorem in Proposition \ref{prop:universal approximation} or from  Theorem \ref{thm:universal under hyp} we know there exists an associated approximation in terms of the signature of the (rough path lifted) price path (either as a linear combination of signature terms, or as a polynomial of a finite number of signature terms). However, while there is no standard way of finding and describing the linear functional $\pi$ in Proposition \ref{prop:universal approximation}, there is much theory available to compute potential sequences of $\{\alpha_m\}_{m\in \cN}$ to obtain a good approximation in Theorem \ref{thm:universal under hyp}. 

We provide now three elementary examples of such approximation choices. 

\begin{ex}[Taylor polynomials]
Suppose the payoff functional $F$ can be identified through hypothesis \ref{hyp} with an analytic function $f\in C^\infty(\RR)$. Then an elementary Taylor expansion of $f(\la \pi,\bX^{\leq \infty}_{0,T}\ra)$ around $0$ yields 
\[
F(\hat{\bX}_{0,\cdot}^{\leq p}) = f(\la \pi,\hat{\bX}^{\leq \infty}_{0,T}\ra)= \sum_{n=0}^\infty \frac{f^{(n)}(0)}{n!}\la \pi,\hat{\bX}^{\leq \infty}_{0,T}\ra^n,
\]
where $f^{(n)}$ denotes the $n$'th derivative of $f$. Moreover, under less restrictive regularity assumptions, we can truncate this sum at any level $k\in \NN$ and explicitly determine the error we make by the formula 
\[
f(\la \pi,\hat{\bX}^{\leq \infty}_{0,T}\ra)=\sum_{n=0}^k\frac{f^{(n)}(0)}{n!}\la \pi,\hat{\bX}^{\leq \infty}_{0,T}\ra^n + R_k(\la \pi,\hat{\bX}^{\leq \infty}_{0,T}\ra)
\]
where there exists some $b$ between $0$ and $\la \pi,\bX^{\leq \infty}_{0,T}\ra$ such that
\[
R_k(x)=\frac{f^{(k+1)}(b)}{(k+1)!}\la \pi,\hat{\bX}^{\leq \infty}_{0,T}\ra^{k+1}.
\]
This yields an analytic expression for the remainder term in Theorem \ref{thm:universal under hyp} when 
$f$ is sufficiently regular.
\end{ex}

\begin{ex}[Hermite polynomials]\label{hermite}
 Let $\pi \in T(\RR^{d+1})^*$ and $K>0$ and suppose we have an option that pays 
    \[
    F(\hat{\bX}_{0,\cdot}^{\leq p})=f(\la \pi,\hat{\bX}_{0,T}^{\leq \infty}\ra)=\max (\la \pi,\hat{\mathbf{X}}^{\leq \infty}_{0,T}\ra - K, 0). 
    \] 
  The function  $\max(x-K,0)$ can be approximated by Hermite polynomials. More precisely, for $n\in \NN_0$ the $n$'th Hermite polynomial is given by 
    \[
    \xi _n(x)=(-1)^n\frac{1}{w(x)}\frac{d^n}{dx^n}w(x) \quad \mathrm{ where} \quad  w(x)=\frac{1}{\sqrt{2\pi}}e^{\frac{-x^2}{2}},
    \]
    is the density of the standard normal distribution $\cN(0,1)$. Note that $\xi_0=1$. Moreover, for the Hilbert space $L^2(\RR,w(x)dx)=:L^2_w$ with inner product 
    \[
    \la g,h\ra_{L^2_w}:=\int_{-\infty}^\infty g(x)h(x)w(x)dx
    \]
    we obtain an orthonormal basis $\{e_n\}_{n=0}^\infty$ given by 
    \[
    e_n(x)=\frac{\xi_n(x)}{\sqrt{n!}}.
    \] In particular, any function $g\in L^2_w$ can be written as 
    \[
    g(x)=\sum_{n=0}^\infty \alpha_ne_n(x),\quad \mathrm{where} \quad  \alpha_n=\la g, e_n\ra_{L^2_w}.
    \]
    As argued in \cite{math9020124}, the function 
    $f(x)=\max(x-K,0)$ for some constant $K\in\RR$, belongs to $L_w^2.$ 
 From this Hermite polynomial expansion,   we have an exact formula for $F(\hat{\bX}_{0,\cdot}^{\leq p})$ given by 
    \[
    F(\hat{\bX}_{0,\cdot}^{\leq p})=\sum_{n=0}^\infty \la f,e_n\ra_{L^2_w} e_n(\la \pi,\hat{\mathbf{X}}^{\leq \infty }_{0,T}\ra).
    \]
    Moreover, we can truncate this sum at any desired level $N$ to reach a suitable approximation by 
    \[
    F_N(\hat{\bX}_{0,\cdot}^{\leq p}):=\sum_{n=0}^N \la f,e_n\ra e_n(\la \pi,\hat{\mathbf{X}}^{\leq \infty}_{0,T}\ra).
    \] 
    
\end{ex}

\begin{ex}[Bernstein approximation]
    The standard choice of approximation of a continuous function by a polynomial is arguably the Bernstein polynomial. Any continuous function $f:[0,1]\rightarrow \RR$ can be approximated arbitrarily well as follows 
\begin{equation*}
    B_n(f)(x)= \sum_{k=1}^n f\left(\frac{k}{n}\right)b_{k,n}(x), \quad \mathrm{{\it where}} \quad     b_{k,n}(x)= {n\choose k}x^{k}(1-x)^{n-k},
\end{equation*}
and $\lim_{n\rightarrow\infty}B_n(f)=f$.
Again, if an option pays $F(\hat{\bX}_{0,\cdot}^{\leq p})$, and  it satisfies Hypothesis \ref{hyp} with a continuous function $f:[0,1]\rightarrow \RR$ and a $\pi\in \cF(\RR^{d+1})$ such that $\langle \pi,\hat{\bX}^{\leq \infty}_{0,T}\rangle\in [0,1]$ and 
\[
F(\hat{\bX}_{0,\cdot}^{\leq p})=f(\langle \pi,\hat{\bX}^{\leq \infty}_{0,T}\rangle), 
\] 
we have the Bernstein approximation
\begin{align*}
    F(\hat{\bX}_{0,\cdot}^{\leq p})&\approx \sum_{k=1}^n f\left(\frac{k}{n}\right)b_{k,n}(\langle \pi, \hat{\bX}^{\leq \infty}_{0,T}\rangle).
\end{align*}

\end{ex}

\section{Stochastic market prices - lifting to rough paths}\label{sec:stochastics}

While the methodologies for numerical approximation of complex derivative payoffs we propose here will be in the spirit of model free finance, we will also connect the results to classical pricing when the underlying stock is assumed to be a semi-martingale. More specifically, we consider the case of an It\^o process of the form 
\begin{equation}\label{eq:ito process}
    \dd Y_t = \mu_t \dd t + \sigma_t \dd B_t, \quad Y_0 =y\in \RR^d.
\end{equation}
Here  $\mu$ and $\sigma$ are square integrable processes and adapted to the filtration $\{\cF_t\}_{t\in [0,T]}$ generated by the Brownian motion  $\{B_t\}_{t\in [0,T]}$. The process $Y$ may be used to model the log-price dynamics of an asset price, or the absolute price, the volatility or any other relevant stochastic asset dynamic. 

In light of the theory for rough paths and signatures presented briefly in the beginning of Section \ref{sec:basics}, it is only natural to ask whether the signature can be constructed above the stochastic process $\{Y\}_{t\in [0,T]}$. More precisely, one wants to make sure that for almost all $\omega\in \Omega$, the following map exists  
\[
Y(\omega)\mapsto \bY^{\leq \infty}(\omega) =\left(1,Y(\omega),\left(\int Y\otimes \dd Y\right)(\omega),\ldots\right)\in V^p(\triangle_T,T((\RR^d))).  
\]
Since $B$ is a Brownian motion, the regularity of $Y$ will be of finite $p$-variation for $p\geq 2$. There is therefore no canonical construction of the iterated integral $\int Y(\omega)\dd Y(\omega)$. However, since $Y$ is a semi-martingale, we can use this probabilistic structure to construct the iterated integrals $   \int Y \dd Y $ as random variables in $L^2(\Omega)$. As is well-known, there exist different choices of constructing this integral as a random variable, with It\^o or Stratonovich integration as the most commonly used. It is up to the application at hand which integral to use for the specific task. Typically, in financial models, It\^o integration is selected as this preserves adaptedness and a martingale structure, necessary for arbitrage-free pricing.  Given the choice of integration, using the Burkholder-Gundy-Davis inequality, one can then apply Kolmogorov's continuity theorem to identify a subset $\Omega^*\subset \Omega$ of full measure such that for each $\omega\in \Omega^* $ there exists a realization of the iterated integral $\left(\int Y\otimes \dd Y\right)(\omega)$, see e.g. \cite[Section 3]{frizhairer}. Moreover, one can verify that this object satisfies Chen's relation 
\begin{multline*}
    Y_{s,u}(\omega)\otimes Y_{u,t}(\omega)
    =\left(\int_s^t Y_{s,r}\otimes \dd Y_r\right)(\omega)
    \\
    -\left(\int_s^u Y_{s,r}\otimes\dd Y_r\right)(\omega)-\left(\int_u^t Y_{u,r}\otimes\dd Y_r\right)(\omega). 
\end{multline*}
We can therefore conclude that $\left(Y(\omega),\left(\int Y\otimes \dd Y\right)(\omega)\right)$ is a $2$-rough path according to Definition \ref{def:rough path}. By applying Theorem \ref{thm:lyons-extension} we know that also the signature $\bY^{\leq  \infty}(\omega)$ exists. 

For practical financial  purposes, we are interested in computing the expected value of multivariate monomials of signature functionals, as will be seen as a crucial component of the pricing approximation. More precisely, the price $p_t$ at time $t\geq 0$ of a contingent claim on a financial asset with payoff at some future time $T\geq t$ can be written as 
\begin{equation*}
    p= \EE[F(\hat{Y})]. 
\end{equation*}
where $F$ is a pay-off functional, possibly dependent on the whole price path $Y$. As previously described, a pay-off functional $F$ is a functional on the rough paths lift $\bY^{\leq p}_{0,\cdot}$  to show the dependence on the  chosen rough paths lift, i.e. stochastic  integration choice. In the next section we will show how this price, given as the conditional expectation,  can be approximated by a sum of different correlators of signature terms of the form 
\begin{equation*}
    \EE[\prod_{i=1}^n \la \pi_i,\hat{\bY}_{0,T}^{\leq \infty}\ra^{m_i}], \quad \mathrm{for} \quad \{\pi_i\}_{i=1}^n \in T((\RR^d)^*),\quad m\in \NN^n. 
\end{equation*}
This is in contrast to the functional approximation considered in \cite{Arribasetal20} where one  computes the complete expected signature $\EE[\hat{\bY}^{\leq \infty}_{0,T}]$, or dynamically as $\mathbb E[\hat{\mathbf{Y}}^{\leq\infty}_{0,T}\,\vert\,\mathcal F_t]$, and then consider $\la \pi, \EE[\hat{\bY}^{\leq \infty}_{0,T}]\ra$. In the setting of an It\^o process $Y$, the latter methodology requires one to solve a (very) high dimensional Kolmogorov equation as described on   \cite{LyonsNi2015}. However, since $Y$ is an It\^o process, then one can show that $t\mapsto \la \pi,\hat{\bY}^{\leq \infty}_{0,t}\ra $ is a real valued It\^o process, as proven in the proposition below.  Thus computing $\EE[f(\la \pi,\hat{\bY}^{\leq \infty}\ra)]$ can be done by standard use of a (low) dimensional Kolmogorov equation, and then one must do this computation for  $f(x)=x^n$ and several different $n$. Indeed, we provide a simple proof of this claim: 

\begin{prop}
Suppose $Y$ is an It\^o process of the form \eqref{eq:ito process} where $\mu\in L^2(\Omega\times [0,T])$ and $\sigma\in L^\infty(\Omega\times [0,T])$ are adapted to  the filtration generated by the Brownian motion. Consider the It\^o lift $Y\mapsto \bY^{\leq \infty}$. Then for any $\pi\in T((\RR^d)^*)$ the process $t\mapsto \la \pi,\bY^{\leq \infty}\ra $ is an It\^o process. 
\end{prop}
\begin{proof}
   Let $\pi=w$ be a single word of length $n$, given of the form $w=\mathbf{i}_1\dots \mathbf{i}_n$ for $\mathbf{i}_j\in \{1,\ldots,d\}$ for all $j=1,\ldots, n$. If $n =1 $, it follows that $t\mapsto \la \mathbf{i}_1,\bY^{\leq \infty}\ra $ is an It\^o process by  definition.  When $n\geq 2$, assume that for all words $w$ of length $n-1$, $\la w,\bY^{\leq \infty}\ra $ is a square integrable and adapted.  We  then use the recursive definition of the signature in \ref{remrecursive} to see that 
   \begin{multline*}
          \la \mathbf{i}_1\dots \mathbf{i}_n,\bY^{\leq \infty}_{s,t}\ra 
          =\int_s^t \la \mathbf{i}_1\dots \mathbf{i}_{n-1},\bY^{\leq \infty}_{s,r}\ra \mu_r^{i_n}\dd r
       \\
       + \sum_{j=1}^d \int_s^t \la \mathbf{i}_1\dots \mathbf{i}_{n-1},\bY^{\leq \infty}_{s,r}\ra \sigma_r^{i_n,j}\dd B^{j}_r.
  \end{multline*}
  To see that this process  is an Ito process we must verify that it is square integrable. Using that $\sigma\in L^\infty(\Omega\times [0,T])$, then by  H\"older's inequality we have that for all $[s,t]\subset [0,T]$
  \begin{align*}
  &\| \la \mathbf{i}_1\dots \mathbf{i}_{n-1},\bY^{\leq \infty}_{s,\cdot}\ra \sigma_\cdot^{i_n,j}\|_{L^2(\Omega\times [s,t])} \\
  &\qquad\qquad\leq \|\sigma^{i_n,j}\|_{L^\infty(\Omega\times [s,t])}  \| \la \mathbf{i}_1\dots \mathbf{i}_{n-1},\bY^{\leq \infty}_{s,\cdot}\ra \|_{L^2(\Omega\times [s,t])}.
  \end{align*}
  By the inductive hypothesis, $ \| \la \mathbf{i}_1\dots \mathbf{i}_{n-1},\bY^{\leq \infty}_{s,\cdot}\ra \|_{L^2(\Omega\times [0,T])}<\infty$ and so the product of the two is square integrable. Adaptedness follows immediately by the inductive hypothesis.  
\end{proof}
For certain choices of $\pi\in T((\RR^d)^{*})$ and assumptions of the underlying stochastic process $Y$, we can compute explicit expressions for the signature moments of the form  $\EE[\la \pi, \bY^{\leq \infty}_{0,T}\ra^n ]$ for $n\in\mathbb{N}$. We illustrate this through some common choices in the following examples. 
\begin{ex}
    Let $\{(B_t^1,B^2_t)\}_{t\in [0,T]}$ be a two dimensional Brownian motion with independent components, and consider the path  $\hat{Y}_t = (t,B_t^1,B_t^2)$ and  the word $\pi= \mathbf{21}-\mathbf{31} $. Recall that this choice of $\pi$ relates to the spread options case considered in Example \ref{ex:asian-quanto}. We are interested in computing the $n$-th moment of $\la \pi,\hat{\bY}^{\leq \infty}\ra $. In this example, it does not matter if we consider the Itô lift or Stratonovich lift. We first see that $\la \pi,\hat{\bY}^{\leq \infty}_{s,t}\ra= \int_s^t B_r^1-B^2_r \dd r$. Note that since $(B^1_t,B^2_t)$ is a normally distributed random variable, their difference $\bar{B}_r:=  B^1_r-B^2_r \sim \cN(0, 2r)$. A simple argument based on the It\^o formula for $t \bar{B}_t$ shows that 
    \begin{equation*}
        \int_0^t \bar{B}_r\dd r = \int_0^t(t-r)\dd \bar{B}_r=\int_0^t (t-r)\sqrt{2}dB_r, 
    \end{equation*}
    where $B_r$ is a standard Brownian motion.
    Thus, using the It\^o isometry, we see that 
    \begin{equation*}
         \EE[\la \pi,\hat{\bY}^{\leq \infty}_{s,t}\ra^2]=\EE\left[\left(\int_0^t \bar{B}_r\dd r\right)^2 \right] = \frac{2}{3}t^3.  
    \end{equation*}
    By Gaussianity, it furthermore follows that, 
    \begin{equation}
    \label{sig-BM-moment-example}
        \EE[\la \pi,\hat{\bY}^{\leq \infty}_{s,t}\ra^n] = \left(\frac{2}{3}t^{3}\right)^\frac{n}{2} (n-1)!! \quad \mathrm{for\,\, even\,\,} n, 
    \end{equation}
    and $\EE[\la \pi,\hat{\bY}^{\leq \infty}_{s,t}\ra^n]=0$ for odd $n$. Here, $m!!=m(m-2)(m-4)\dots 3\cdot 1$ for $m\in\mathbb N$ being an odd number.
\end{ex}
The above example shows that the moments of words applied to signatures are very easy to calculate when the path is the time-extended Brownian motion. \cite[Section 4.5]{C-SF-FS} derive an explicit formula for the expected signature of the time-extended path of a $d$-dimensional Brownian motion. According to formula (4.18) in \cite{C-SF-FS}, one has
\begin{equation}
\label{expect-sig-BM-christa}
\langle w,\mathbb E[\bY^{\leq N}_{0,t}]\rangle=\frac{(t/2)^{n/2}}{(n/2)!}\prod_{k=0}^{n/2-1}\mathbf{1}_{i_{n-2k}=i_{n-2k-1}}
\end{equation}
for words $w=e_{i_1}\otimes..\otimes e_{i_n}$ where $\{i_1,\ldots,i_n\}\in\{1,\ldots,d\}^n$, even $n\in\mathbb N$ with $n\leq N$ and $e_i$ is the $i$th unit vector in $\mathbb R^d$. The tensor products are here interpreted as the non-symmetric ones and the iterated integrals in the signature are interpreted in the Stratonovich-sense. We recall that in a financial setting however, we are mostly interested in martingale structures, typically guaranteed with the It\^o lift, and must there also compute certain It\^o-Stratonovich corrections that we do not consider further here. 
 When $n$ is odd, the expected signature is zero. An extension of \eqref{expect-sig-BM-christa} to correlated Brownian motions are found in \cite[Thm. A.1]{C-G-SF-SIFIN}. Equation \eqref{expect-sig-BM-christa} is of similar complexity as \eqref{sig-BM-moment-example} in the example above. However, we see that to compute the expected signature, we need to find the whole representation of the expected signature (up to depth $N$) {\it before} applying it to words, whereas in our approach, we
first identify the different terms in the signature which we need according to the words we are given, and {\it then} compute the expected moments in question. The words in the former approach might be very long as we have converted the moments into linear representations, and thus $N$ becomes big. 

Towards signature approximations of complex derivative prices,  we will need to investigate correlators of functionals of the signature. More precisely, let $F\in C(\hat{V}^p([0,T],T^{\lfloor p\rfloor}(\RR^{d+1}))$ and assume Hypothesis \ref{hyp} holds for the weighted words $\{\pi_i\}_{i=1}^n\in T(\RR^{d+1}).$ Then for a multi-index $m\in \NN^n$ we define the correlator $\rho_m$ by
\begin{equation}\label{eq:correlator}
    \rho_m = \EE\left[\left(\la \pi_1,\bY^{\leq \infty}_{0,T}\ra ,\ldots,\la \pi_n,\bY^{\leq \infty}_{0,T}\ra \right)^m\right],
\end{equation}
where we recall that for a vector $x\in \RR^n$ we define $x^m = x^{m_1} \dots x^{m_n}$.  Correlators appear in statistical turbulence theory as interesting objects to study, see \cite{BNBV-ambit}. We also mention \cite{BL-poly-corr} for correlators applied to financial derivatives pricing along with polynomial processes.
Computing these correlators for some arbitrary sequence $\{\pi_i\}_{i=1}^n$ can be challenging, and if $n$ and $d$ gets large, might even become unfeasible. However, for certain linear functionals $\pi_i$, typically given as short words or even single letters, the computation may be analytically tractable by invoking stochastic structures, even for large $n$.

In contrast, the signature methodologies presented in \cite{LyonsNejadAribas2019,Arribasetal20} would require one to compute the expected signature term $\la \phi,  \EE[\bY^{\leq \infty}_{0,T}]\ra$ for some $\phi\in T((\RR^d)^*)$ which potentially becomes a very long sum of very large words. Indeed, given that $\bY^{\leq \infty}$ is a geometric signature,  as already discussed in Remark \ref{rem:computational efficiency}, using the shuffle product from Lemma \ref{lemma:shuffle},  it is possible to find a $\phi\in T((\RR^d)^*)$ such that
\[
\left(\la \pi_1,\bY^{\leq \infty}_{0,T}\ra ,\ldots,\la \pi_n,\bY^{\leq \infty}_{0,T}\ra \right)^m = \la \phi,\bY^{\leq \infty}_{0,T}\ra.
\]
Computationally, it will be more expensive to compute the right hand side than the left hand side, since one would need to compute the complete signature up to a very high degree (i.e. the length of the longest single word in $\phi$). 
Furthermore, we have that the correlator is given by the expectation of the right hand side $\la \phi,\bY^{\leq\infty}_{0,T}\ra$, i.e.,  $\rho_m= \EE[\la \phi,\bY^{\leq \infty}_{0,T}\ra]$, and thus computing these expected signature terms will be challenging. Even in the setting where the underlying path solves an It\^o SDE, one is required to solve a Fokker-Planck equation with values in the tensor algebra $T((\RR^d))$, which is numerically very challenging (see e.g. \cite{LyonsNi2015}).

\section{Approximation of exotic derivatives}\label{sec:derivatives}

We are now ready to present the core results of this article, namely, an approximation formula for exotic, path dependent, financial derivatives. 
To this end, we begin with an assumption on the probability spaces we work with for the admissible price paths. 

\begin{hyp}[Market prices and probability measures]\label{hyp:market prices}
We consider a complete stochastic basis $(\Omega,\cF, \{\cF_t\}_{t\in[0,T]},\PP)$ supporting  market prices as measurable maps from $\Omega$ to $\hat{V}^{p}([0,T];T^{\lfloor p\rfloor }(\RR^{d+1}))$, with the property that for any $\tilde{\epsilon}>0$ there exists a compact subset $\cK_\epsilon \subset \hat{V}^{p}([0,T];T^{\lfloor p\rfloor }(\RR^{d+1})) $ such that 
\[
\PP( \hat{\bX}^{\leq p}_{0,\cdot} \in \cK_\epsilon)\geq 1-\tilde{\epsilon}.
\]
\end{hyp}

\begin{rem}
    Let $E$ be a separable and complete metric space. Then every Borel probability measure $\mu$ on $E$ is tight; that is, for every $\epsilon > 0$, there exists a compact set $\cK_\epsilon \subseteq E$ such that $\mu(\cK_\epsilon) > 1 - \epsilon$, see \cite{Linde1986}. Consequently, if $E := \hat{V}^p([0,T]; T^{\lfloor p \rfloor}(\mathbb{R}^d))$ were separable, the hypothesis above would be unnecessary. 
    
    However, as shown in \cite[Sec. 8]{Friz_Victoir_2010}, the space $V^{p}([0,T], G^N(\RR^d))$ is in general \textit{not} separable.  In contrast, the closure of smooth paths with values in the $N$-step free Lie group $\RR^d$ in $p$-variation norm, denoted by $\overline{C^\infty([0,T],G^N(\RR^d))}^p$, is separable.
    Note also that we have the inclusion of spaces 
    \[
        \bigcup_{1\leq q<p}V_c^{q}([0,T], G^N(\RR^d))\subseteq \overline{C^\infty([0,T],G^N(\RR^d))}^p \subseteq V_c^{p}([0,T], G^N(\RR^d))
    \]
    Since $E$ is not separable, we must explicitly assume that all distributions $\mathbb{P}_{\hat{\mathbb{X}}^{\leq p}}$ on $E$ are tight, or equivalently, that they are Radon measures (finite tight Borel measures). In the non-separable setting of $E$, Radon measures - characterized by their separable image - are the "right" type of measures to consider. For further details, see \cite{Linde1986} or \cite{Billingsley1986}.
\end{rem}
The above market prices and probability measures provides a broad framework for pricing. In the following we will not deal with the problem of risk neutral prices, and rather refer the reader to \cite{arribas2018derivativespricingusingsignature,Arribasetal20} for a discussion on this point. For the purpose here, the reader may assume that the probability measure $\PP$ chosen is a risk neutral measure in the context of the given pricing problem.  
\begin{thm}\label{thm:price approx}
Let $(\Omega, \cF, \{\cF_t\}_{t\in [0,T]}, \PP)$ be a filtered probability space satisfying Hypothesis \ref{hyp:market prices}, and suppose the price of a financial derivative, denoted by $p$, can be represented by an adapted payoff functional $F$ acting on the set of random price paths $L^k(\Omega;\hat{V}^p([0,T];T^{\lfloor p \rfloor }(\RR^{d+1})))$ for some sufficiently large $k$ (see Remark \ref{rem:choice of k}), and is given by 
\begin{equation}\label{eq:payoff}
    p = \EE[F(\hat{\bX}_{0,\cdot}^{\leq p})]. 
\end{equation}
Furthermore, suppose for all $t\in [0,T]$, $\EE[|F(\hat{\bX}_{0,\cdot}^{\leq p})|^q]=M<\infty$  for some $q\geq 1$. 
Then for any $\epsilon>0 $ there exists a finite set $\cN\subset \NN^d$ and  a sequence of numbers $\{\alpha_m\}_{m\in \cN}$ such that 
\begin{equation*}
    \left|p-\sum_{m\in\cN } \alpha_m \rho_m\right|<\epsilon.
 \end{equation*}
Where $\rho_m$ denotes the signature correlator  of $F$   from \eqref{eq:correlator} for the stochastic process $\hat{\bX}^{\leq \infty}$. 
\end{thm}
\begin{rem}\label{rem:choice of k}
    Here we consider payoff functionals as functionals acting on random variables  with finite $k$-moments. The requirement on $k$ will depend on the approximation accuracy that is desired, since finiteness of the correlators $\rho_m$ is only guaranteed from the moments of the price paths.  
\end{rem}
\begin{proof}
We begin to observe that for some suffieciently large compact subset  $\mathcal K_{\tilde{\epsilon}} \subset \hat{V}^p([0,T];T^{\lfloor p \rfloor }(\RR^{d+1})) $  we have 
 \begin{align*}
         |\EE[F(\hat{\bX}_{0,\cdot}^{\leq p})]| &= |\EE[F(\hat{\bX}_{0,\cdot}^{\leq p})\mathbf{1}_{\mathcal K_{\tilde{\epsilon}}}|]|+ \EE[|F(\hat{\bX}_{0,\cdot}^{\leq p})\mathbf{1}_{\mathcal K_{\tilde{\epsilon}}^c}|] 
         \\
         &\leq M^{\frac
    {1}{    q}} \tilde{\epsilon}^{\frac{q-1}{q}} + \EE[F(\hat{\bX}_{0,\cdot}^{\leq p})\mathbf{1}_{\mathcal K_{\tilde{\epsilon}}}]. 
    \end{align*}
    In the last inequality we have applied H\"olders inequality to the product $\mathbf{1}_{\mathcal K_{\tilde{\epsilon}}}F(\hat{\bX}_{0,\cdot}^{\leq p})$, invoking the bound on the $q$-moment of $F(\hat{\bX}_{0,\cdot}^{\leq p})$ as well as Hypothesis \ref{hyp:market prices} on the probability measure to get that $(\EE [|\mathbf{1}_{\mathcal K_{\tilde{\epsilon}}}|^{\frac{q}{q-1}}])^{\frac{q-1}{q}}\leq \tilde{\epsilon}^{\frac{q-1}{q}} $.
Now,  apply Theorem \ref{thm:universal general} to the payoff functional $F(\hat{\bX}_{0,\cdot}^{\leq p})$ and the expectation acting on the corresponding signature polynomial from Theorem \ref{thm:universal general} then yields the correlators from \eqref{eq:correlator} applied to the signature $\hat{\bX}^{\leq \infty}_{0,T}$. 
Since $\tilde{\epsilon}$ can be chosen arbitrarily small by assumption (just choose $\cK_\epsilon$ larger), and likewise for the universal approximation, the proof is complete. 
\end{proof}

The next corollary  enables further simplifications for pricing approximation by invoking Hypothesis \ref{hyp} on the payoff functional. 
\begin{cor} \label{thm5.3}
Suppose the price of a financial derivative, denoted by $p$, can be represented by a payoff functional $F$ acting on the set of admissible price paths $\in \hat{V}^p([0,T];T^{\lfloor p \rfloor }(\RR^{d+1}))$, and is given by \eqref{eq:payoff}.
Furthermore, suppose that for $F$ Hypothesis \ref{hyp} holds for some function $f:\RR^n \rightarrow \RR$. 
Then for any $\epsilon>0 $ there exists a finite set $\cN\subset \NN^d$ and  a sequence of numbers $\{\alpha_m\}_{m\in \cN}$ such that 
\begin{equation*}
    \left|p-\sum_{m\in\cN } \alpha_m \rho_m\right|<\epsilon,
 \end{equation*}
\end{cor} 
\begin{proof}
This follows from the universal approximation in Theorem  \ref{thm:universal under hyp}, together with a similar probabilistic analysis as in the proof of Theorem \ref{thm:price approx}. 
\end{proof}

We illustrate the above corollary by considering a specific choice of functional, common in financial practice, namely the max function acting on the signature; the functionals for certain Asian options. 

\begin{ex}[Simple Asian option]\label{ex:asian}
We will in this example consider a very basic Asian option, and show how moments of the signature can be used to approximate this price. Of course, there are well known formulas and approximations (see e.g. \cite{GemanYor}) for the price of an Asian option in the It\^o setting, but we believe that this is instructive. In incomplete markets, one could also imagine that $X$ is not an It\^o process, and therefore the following expansion could still be an interesting pricing technique.

Let us consider the standard Asian call option payoff function 
\[
F(\hat{\bX}_{0,\cdot}^{\leq p})=f(\la \mathbf{21},\hat{\bX}_{0,T}^{\leq \infty}\ra)=\max\left(0,\int_0^TX_sds-K\right)
\]
A smooth approximation for the max-function is given by $\bar{f}(x)=x\sigma(Nx)=\frac{x}{1+e^{-Nx}}$, where $\sigma(x)=\frac{1}{1+e^{-x}}$ is the well-known sigmoid function and for some large enough $N\in\NN$. The figure below shows how this function, $\bar{f}$, resembles the max-function.
\begin{figure}[h]\label{fig}
\centering
\caption{Plot of $\bar{f}(x)=\frac{x}{1+e^{-Nx}}$  }
\includegraphics[width=0.75\textwidth]{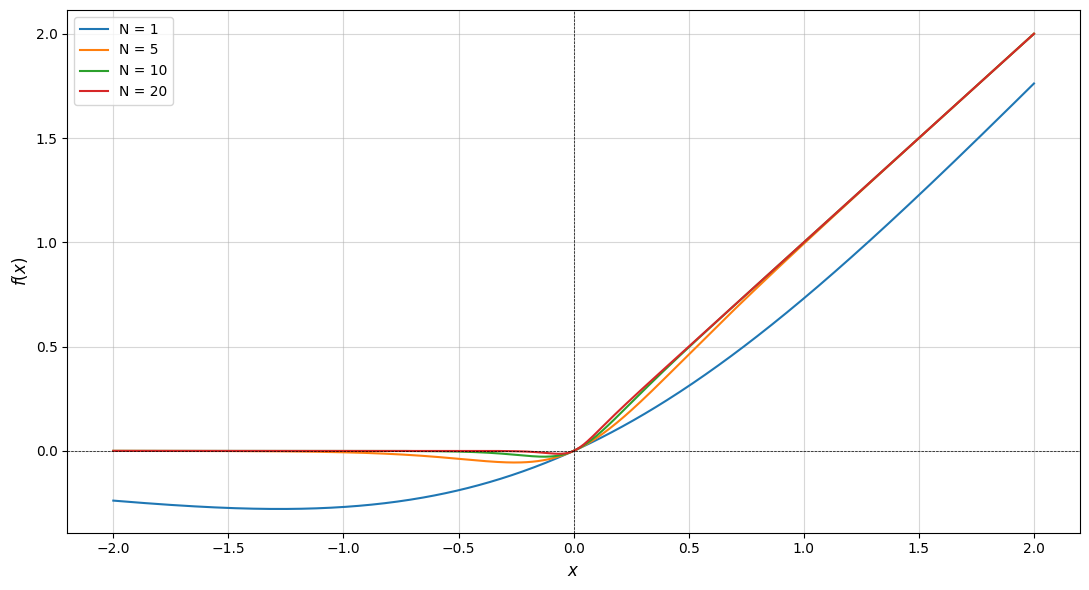}
\end{figure}

We have a Maclaurin series for the Sigmoid function $\sigma(Nx)=\frac{1}{1+e^{-Nx}}$, and then multiplying by $x$  we get : 
\begin{equation}
    \bar{f}(x) = \sum_{n=0}^\infty\frac{(-1)^nE_n(0)}{(2n)!}N^{n}x^{n+1}, 
\end{equation}
where  $E_n(x)$ is the Euler polynomial.
Truncating this approximation at level $M$, one get a price approximation for 
\[
p =\EE[F(\hat{\bX}_{0,\cdot}^{\leq p})]\approx \sum_{n=0}^M \frac{(-1)^nE_n(0)}{(2n)!}N^{n}\EE\left[\la \mathbf{21},\hat{\bX}_{0,T}^{\leq \infty}\ra^{n+1}\right]. 
\]
Providing further theoretical convergence rates can be done under certain assumptions on the moments $\EE[\la \mathbf{21},\hat{\bX}_{0,T}^{\leq \infty}\ra^n]$, but will not be further dealt with here. However, we believe that the example highlights how moments of signatures can be used in derivative price approximation. 
\end{ex}

\subsection{Pricing exotic derivatives in electricity markets}

We consider here some cases of interest in pricing and valuation in electricity markets. 

The so-called {\it quality factor} is used to assess the profitability of renewable power production such as solar or wind. It measures the income relative to a plant with fixed base load price producing the same volume. The quality factor is defined as
\begin{equation}
    Q=\frac{\int_0^TV_sP_s\,ds}{\int_0^TV_s\,ds\frac1T\int_0^TP_s\,ds}\,.
\end{equation}
Here, $V_t$ is the volume power produced at time $t$ and $P_t$ is the spot power price. A natural question is to ask what is the expected quality factor, i.e., $\mathbb E[Q]$. The expectation is either under the risk-adjusted probability or the market probability. 

The volume produced $V_t$ is given by the installed capacity $c$ (measured in megawatt (MW)) times the capacity factor $C_t$. The process $(C_t)$ takes values between 0 and 1, measuring the amount of production from solar or wind in a power plant of capacity 1 MW. We also introduce a 
maximal power price for the market, denoted $P_{\infty}$, which can be an upper limit we believe never will be exceeded in practice. Hence, with $V_t=c C_t$ and $S_t=P_t/P_{\infty}$ we get
\begin{align*}
Q&=\frac{\frac1T\int_0^TC_s S_s\,ds}{\frac1T\int_0^TC_s\,ds\frac1T\int_0^TS_s\,ds}\\
&=\frac1T\int_0^TC_sS_s\,ds\frac{1}{1+(\frac1T\int_0^TC_s\,ds-1)}\frac{1}{1+(\frac1T\int_0^TS_s\,ds-1)}\,.
\end{align*}
But since $C_s,S_s\in(0,1)$, we have that $(1/T)\int_0^TC_s\,ds-1\in(-1,0)$ and similarly $(1/T)\int_0^TS_s\,ds-1\in(-1,0)$, which yields that
\begin{align*}
Q&=\frac1T\int_0^TC_sS_s\,ds\sum_{m=0}^{\infty}(-1)^m\left(\frac1T\int_0^TC_s-1\,ds\right)^m \\
&\qquad\times\sum_{n=0}^{\infty}(-1)^n\left(\frac1T\int_0^TS_s-1\,ds\right)^n\,.
\end{align*}
Introduce the time-enhanced process $Y$ where $Y_t=(t,C_t,S_t)$. 
Since \[C_sS_s=(C_s-C_0)(S_s-S_0)+C_0(S_s-S_0)+S_0(C_s-C_0)+C_0S_0\] and \[\langle\mathbf{2},\hat{\mathbf{Y}}_{0,s}^{\leq\infty}\rangle=C_s-C_0=C_{0,s},\quad \langle\mathbf{3},\hat{\mathbf{Y}}_{0,s}^{\leq\infty}\rangle=S_s-S_0=S_{0,s},\]

it follows from the shuffle property in Lemma \ref{lemma:shuffle} that 
\begin{align*}
\int_0^TC_sS_s\,ds&=\int_0^TC_{0,s}S_{0,s}\,ds+C_0\int_0^TS_{0,s}\,ds+S_0\int_0^TC_{0,s}\,ds+C_0S_0\int_0^Tds\\&=\langle\pi_1,\hat{\mathbf{Y}}_{0,T}^{\leq\infty}\rangle\,,
\end{align*}
where  
\[ 
\pi_1 = (\mathbf{2}\shuffle\mathbf{3})\mathbf{1} + C_0\mathbf{3}\mathbf{1} + S_0\mathbf{2}\mathbf{1}+C_0S_0\mathbf{1}, 
\]

as long as $Y$ is a geometric rough path. Moreover,
we notice that $\int_0^TC_s-1\,ds=\langle \pi_2,\hat{\mathbf{Y}}_{0,T}^{\leq\infty}\rangle$ and
$\int_0^TS_s-1\,ds=\langle \pi_3,\hat{\mathbf{Y}}_{0,T}^{\leq\infty}\rangle$
with $\pi_2=\mathbf{2}\mathbf{1} - C_0\mathbf{1}$ and 
$\pi_3=\mathbf{3}\mathbf{1}-S_0\mathbf{1}$, respectively. Hence, after truncation, we compute an approximation of the expected value of the
quality factor by
\begin{equation}
\label{eq:approx-qf}
\mathbb{E}[Q]\approx\frac1T\sum_{m=0}^M\sum_{n=0}^N(-T)^{-(n+m)}\mathbb{E}\left[\langle\pi_1,\hat{\mathbf{Y}}_{0,T}^{\leq\infty}\rangle\langle \pi_2,\hat{\mathbf{Y}}_{0,T}^{\leq\infty}\rangle^m\langle \pi_3,\hat{\mathbf{Y}}_{0,T}^{\leq\infty}\rangle^n\right]
\end{equation}
We can simulate the functionals inside the above expectation given 
stochastic models of $C$ and $S$. Notice that we only need to have available simulations of the three functionals in order to compute the expectations by Monte Carlo for any orders of $m$ and $n$. If we appeal to Lemma \ref{lemma:shuffle} we can use the shuffle property to re-state the signature functionals inside the expectation to 
$\langle\pi,\hat{\mathbf{Y}}_{0,T}^{\leq\infty}\rangle$, however, $\pi$ will depend on $m$ and $n$ and therefore we would need to consider higher and higher signatures in the calculation. This shows the power in our approach.

\begin{rem}
    In the above discussion of the computation of the expected quality factor, we have to assume $Y$ being a geometric rough path in order to re-express a product between the capacity factor and price processes $C_sP_s$ as a functional on the signature. The capacity factor process is derived from wind speeds or solar irradiation, where empirical studies indicate that higher-order continuous time autoregressive (CAR) processes are suitable (see e.g. \cite{BSB-wind} for wind and \cite{GLB-solar} for solar). The CAR-processes are of order 2 or higher, implying that the paths are continuously differentiable and thus of finite 1-variation. On the other hand, different studies point to power spot prices being sums of Ornstein-Uhlenbeck processes (see \cite{LS2002}), or even fractional models with Hurst parameters less than 0.5 (see \cite{Bennedsen}), and therefore may have regularity at most as  Brownian motion. We notice however, that power spot prices (which is what the process $P$ models) are by definition only available on a discrete time grid (typically of hourly granularity), and hence we can in principle imagine paths which are of high regularity when considering continuous-time paths.         
\end{rem}

Recall Example \ref{ex:asian-quanto}. Quanto options are options with a product payoff on two underlying assets, typically being a call or put on price and a volume variable. The volume variable is indicating production indirectly, for example through temperature (which controls the demand for power) or wind speed (which controls the amount of renewable wind power that can be generated). Let now $V_t$ be the {\it volume}-variable (temperature, wind...). In power, a typical option is a call on average volume over a period, and a put on the price average over the same period.\footnote{This gives an insurance against too low prices when renewable power production is high.}, 
$$
\max\left(\frac1T\int_0^TV_s\,ds-K,0\right)\max\left(L-\frac1T\int_0^TP_s\,ds,0\right)
$$
In general, this can be expressed as 
$$
F(Y)_T=f_{-K}(\langle w_2,\hat{\mathbf Y}^{\leq\infty}_{0,T}\rangle)f_{L}(\langle w_3,\hat{\mathbf Y}^{\leq\infty}_{0,T}\rangle)
$$
where $f_a(x)=\max(0,x+a)$ for some $a\in \RR$,  with words $w_2=T^{-1}\mathbf{2}\mathbf{1}+T^{-1}V_0\mathbf{1}$ and $w_3=-T^{-1}\mathbf{3}\mathbf{1}-T^{-1}P_0\mathbf{1}$.
Here, the time-enhanced process is $Y_t=(t,V_t,P_t)$. If we have power series expansions of $f$ and $g$ available, we can find an approximation of the price expressed by the risk-adjusted expectation by computing terms of the kind 
$$
\mathbb{E}\left[\langle w_2,\hat{\mathbf{Y}}_{0,T}^{\leq\infty}\rangle^m\langle w_3,\hat{\mathbf{Y}}_{0,T}^{\leq\infty}\rangle^n\right]\,.
$$
These expected values are simplified versions of the expectation for the quality factor in \eqref{eq:approx-qf} discussed above.  
The payoff of quanto options motivates studying the following general structure: let $h:\mathbb R^d\rightarrow\mathbb R$ be some measurable function, and consider the payoff 
$$
F(Y)_T=h(\langle\pi_1,\hat{\mathbf Y}^{\leq\infty}_{0,T}\rangle,\ldots,\langle\pi_d,\hat{\mathbf Y}^{\leq\infty}_{0,T}\rangle)
$$
for words $\pi_1,\ldots,\pi_d$ and path $Y$. Indeed, the quality factor is itself an example of a specification of an $h$. The time-enhanced price path may also consist of more variables that $V$ and $P$. For example, one could have an option settled on temperature, wind and price.

\begin{ex}
In this example we will illustrate how specific structural assumptions on the driving underlying price processes can be used to make price approximations expcicit by leveraging moment computations.  Consider two correlated Ornstein-Uhlenbeck processes
\[
\begin{aligned}
Y^1_t&=-a_1Y^1_tdt+\sigma_1 dB^1_t
\\
Y^2_t&=-a_2Y^2_tdt+\sigma_2( \rho \dd B_t^1+\sqrt{1-\rho^2} \dd B^2_t)
\end{aligned}
\]
for two independent Brownian motions $B^1$ and $B^2$, and $\rho\in (-1,1)$ is the correlation coefficient. Solving explicitly, we see that the difference $\bar{Y}_t= Y^1_t-Y^2_t$ is given by 
\begin{multline}
    \bar{Y}_t= e^{-a_1t}Y^1_0-e^{-a_2t}Y^2_0
    + \int_0^t \sigma_1 e^{-a_1(t-s)}-\rho \sigma_2 e^{-a_2(t-s)}\dd B^1_s
   \\  - \int_0^t e^{-a_2(t-s)}\sigma_2\sqrt{1-\rho^2 }\dd B_s^2.
\end{multline}
We then see that $\bar{Y}$ is normally distributed with mean $e^{-a_1t}Y^1_0-e^{-a_2t}Y^2_0$, and second moment given by 
\begin{multline}
    \EE[\bar{Y}^2_t] = \int_0^t \sigma_1^2 e^{-2a_1(t-s)}+\sigma_2^2 e^{-2a_2(t-s)}
    -2\sigma_1\rho \sigma_2  e^{-(a_1+a_2)(t-s)}\dd s.
\end{multline}
Now, just as in Example \ref{sig-BM-moment-example}, we again consider the pairing $\la \pi,\hat{\bY}^{\leq\infty}_{0,T}\ra=\int_{0}^t\bar{Y}_sds=:Z_t,$ where
$\pi=\mathbf{21}-\mathbf{31}+\bar{Y}_0\mathbf{1}.$ A simple use of Fubini yields that 
\begin{multline}
    Z_t = \frac{Y^1_0}{a_1} (1 - e^{-a_1t}) - \frac{Y^2_0}{a_2} (1 - e^{-a_2t}) 
+ \int_0^t \frac{\sigma_1}{a_1} \left( 1 - e^{-a_1(t-u)} \right) \, \mathrm{d}B^1_u 
\\
- \int_0^t \frac{\rho \sigma_2}{a_2} \left( 1 - e^{-a_2(t-u)} \right) \, \mathrm{d}B^1_u 
- \int_0^t \frac{\sigma_2 \sqrt{1-\rho^2}}{a_2} \left( 1 - e^{-a_2(t-u)} \right) \, \mathrm{d}B^2_u.
\end{multline}
Using Itô-isometry, we get the mean and second moment of $Z_t=\la \pi,\hat{\bY}^{\leq\infty}_{0,T}\ra$ is given by 
\begin{align*}
\mu_Z:= \mathbb{E}[Z_t] &= \frac{Y^1_0}{a_1}(1 - e^{-a_1t}) - \frac{Y^2_0}{a_2}(1 - e^{-a_2t}).
\\
\sigma^2_Z:=\mathbb{E}[Z_t^2]& = 
\left(\frac{Y^1_0}{a_1}(1 - e^{-a_1t}) - \frac{Y^2_0}{a_2}(1 - e^{-a_2t})\right)^2 \\
&+ \int_0^t \left(\frac{\sigma_1}{a_1}(1 - e^{-a_1(t-u)})\right)^2 \, \mathrm{d}u \\
&+ \int_0^t \left(\frac{\rho \sigma_2}{a_2}(1 - e^{-a_2(t-u)})\right)^2 \, \mathrm{d}u \\
&+ \int_0^t \left(\frac{\sigma_2 \sqrt{1-\rho^2}}{a_2}(1 - e^{-a_2(t-u)})\right)^2 \, \mathrm{d}u.
\end{align*}
Moreover, since $Z_t=\la \pi,\hat{\bY}^{\leq\infty}_{0,T}\ra$ is Gaussian, the higher order moments become
\[
\EE[\la \pi,\hat{\bY}^{\leq\infty}_{0,T}\ra^n]=\EE[Z_t^n]= \sigma_Z^n (-i\sqrt{2})^{n} U\left(-\frac{p}{2},\frac{1}{2},-\frac{1}{2} \left(\frac{\mu_Z}{\sigma_Z}\right)^2 \right), 
\]
where $U$ is the confluent hyper-geometric function.
Now,  the correlator 
is given by $\rho_{n}=\EE\left[\la \pi,\hat{\bY}^{\leq\infty}_{0,T}\ra^{n}\right]$, hence from
Corollary \ref{thm5.3}, we know there exist an $N\in\NN$ and $\{\alpha_n\}_{n=0}^N\subset \RR$ such that
    \[
    p:=\EE\left[F(\hat{\bY}^{\leq p}_{0,\cdot})\right]\approx \sum_{n=0}^N\alpha_n\rho_n=\sum_{n=1}^N \alpha_n 
 \sigma_Z^n (-i\sqrt{2})^{n} U\left(-\frac{p}{2},\frac{1}{2},-\frac{1}{2} \left(\frac{\mu_Z}{\sigma_Z}\right)^2 \right).
    \]
    Consider the setting where the expected functional  is an Asian option, such as analyzed in \ref{ex:asian}. One can then use the coefficients found there to get an explicit formula for the approximation of an Asian spread option,
 in the case when the underlying processes are assumed to be given by Ornstein-Uhlenbeck processes. 
\end{ex}

\section{Conclusion}\label{sec:conclussion}

In this work, we  develop a new type of universal approximation theorem for non-geometric rough paths, addressing the practical challenges of financial applications that naturally involve Itô integration. The results presented provide a robust framework for advancing derivatives pricing methodologies in financial markets, ensuring both computational efficiency and theoretical rigor. 
By introducing a polynomial-based approximation framework, we demonstrated how complex payoff functionals for financial derivatives can be efficiently represented and approximated using signature terms. This approach bridges the gap between rough path theory and the practical requirements of financial markets, providing a robust tool for pricing exotic derivatives and path-dependent contracts.

Our results highlight the versatility of signatures in capturing the intricacies of stochastic paths while enabling computational efficiency and how this can be used in finance. The proposed framework not only broadens the scope of universal approximation beyond geometric rough paths but also lays a foundation for further exploration of functional approximation in stochastic finance. 
 Specifically, it builds further on the research developed in \cite{LyonsNejadAribas2019,Arribasetal20,arribas2018derivativespricingusingsignature}, and provides a new perspective in the It\^o setting
The methodology for universal approximation here seems also promising in the context of the Volterra signature induced from the analysis in \cite{Harang_Tindel21}, and this is currently something we are working on. Several new applications of this method seems promising.

\bibliographystyle{alpha} % We choose the "plain" reference style

\end{document}